\documentclass[12pt, reqno]{amsart}
\usepackage[scaled]{helvet}
\usepackage{mathtools}
\mathtoolsset{showonlyrefs}
\usepackage{etoolbox}
\usepackage[utf8]{inputenc}
\usepackage[T1]{fontenc}

\usepackage{amsmath,amssymb,amsfonts,mathrsfs,amsbsy}
\usepackage{relsize}
\usepackage{enumitem}
\usepackage{anyfontsize}
\usepackage{tikz,xcolor}
\usepackage{bm}
\usepackage{comment}
\usepackage[margin=1in]{geometry}
\definecolor{lava}{rgb}{0.81,0.06,0.13}
\definecolor{Cblue}{rgb}{0.50,0.85,0.85}
\definecolor{lime}{HTML}{A6CE39}

\usepackage[colorlinks,linkcolor=lava,citecolor=blue,urlcolor=Cblue,hypertexnames=false]{hyperref}

\usepackage{caption}  
\usepackage{subcaption}  
\usepackage{float}

\definecolor{Cgrey}{rgb}{0.85,0.85,0.85}
\definecolor{Cblue}{rgb}{0.50,0.85,0.85}
\definecolor{Cred}{rgb}{1,.2,.4}
\definecolor{fancy}{rgb}{0.10,0.85,0.10}
\definecolor{mygreen}{rgb}{0.01,0.6,0.2}
\definecolor{tealgreen}{rgb}{0.0, 0.51, 0.5}
\definecolor{tangerine}{rgb}{0.95, 0.52, 0.0}
\definecolor{saffron}{rgb}{0.96, 0.77, 0.19}
\definecolor{mint}{rgb}{0.24, 0.71, 0.54}
\definecolor{lincolngreen}{rgb}{0.11, 0.35, 0.02}
\definecolor{lava}{rgb}{0.81, 0.06, 0.13}
\definecolor{lasallegreen}{rgb}{0.03, 0.47, 0.19}
\definecolor{mahogany}{rgb}{0.75, 0.25, 0.0}
\definecolor{electricultramarine}{rgb}{0.25, 0.0, 1.0}
\definecolor{mypink1}{rgb}{0.858, 0.188, 0.478}
\definecolor{mypink2}{RGB}{219, 48, 122}
\definecolor{mypink3}{cmyk}{0, 0.7808, 0.4429, 0.1412}
\definecolor{mygray}{gray}{0.6}
\definecolor{venetianred}{rgb}{0.78, 0.03, 0.08}
\definecolor{sapphire}{rgb}{0.03, 0.15, 0.4}
\definecolor{utahcrimson}{rgb}{0.83, 0.0, 0.25}
\definecolor{trueblue}{rgb}{0.0, 0.45, 0.81}
\definecolor{carminered}{rgb}{1.0, 0.0, 0.22}
\definecolor{cobalt}{rgb}{0.0, 0.28, 0.67}
\definecolor{cornflowerblue}{rgb}{0.39, 0.58, 0.93}
\definecolor{falured}{rgb}{0.5, 0.09, 0.09}

\usepackage{enumitem}
\usepackage{bm}

\usepackage{tikz}

\newtheorem{theorem}{Theorem}[section]
\newtheorem{lemma}[theorem]{Lemma}
\newtheorem{proposition}[theorem]{Proposition}
\newtheorem{remark}[theorem]{Remark}
\newtheorem{definition}[theorem]{Definition}

\newcommand{\R}{\mathbb{R}}
\newcommand{\N}{\mathbb{N}}
\newcommand{\Om}{\Omega}

\numberwithin{equation}{section}

\numberwithin{bbb}{section}
\hypersetup{unicode=true}

\definecolor{lime}{HTML}{A6CE39}
\DeclareRobustCommand{\orcidicon}{%
  \begin{tikzpicture}
    \draw[lime, fill=lime] (0,0) circle [radius=0.16]
      node[white] {\scriptsize ID};
    \draw[white, fill=white] (-0.0625,0.095) circle [radius=0.007];
  \end{tikzpicture}%
  \hspace{-2mm}%
}
\foreach \x in {A, B, C}{%
  \expandafter\xdef\csname orcid\x\endcsname{%
    \noexpand\href{https://orcid.org/\csname orcidauthor\x\endcsname}{\noexpand\orcidicon}%
  }%
}

\newcommand{\myskip}[1]{}
\title[Brinkman--Biofilm--Nutrient]{Well-Posedness of a Coupled Brinkman--Biofilm--Nutrient System with Volume-Fraction Constraints}
\author[A. Alhammali \& M. Majdoub]{Azhar Alhammali\orcidB{} and Mohamed Majdoub\orcidA{}}

\address{Department of Mathematics, College of Science, Imam Abdulrahman Bin Faisal University, P. O. Box 1982, Dammam, Saudi Arabia\newline
Basic and Applied Scientific Research Center, Imam Abdulrahman Bin Faisal University, P.O. Box 1982, 31441, Dammam, Saudi Arabia}

\email[A. Alhammali]{{\tt aalhammali@iau.edu.sa}}
\email[M. Majdoub]{{\tt mmajdoub@iau.edu.sa}}
\email[M. Majdoub]{{\tt med.majdoub@gmail.com}}
\email[M. Majdoub]{{\tt mohamed.majdoub@fst.rnu.tn}}

\subjclass[2020]{Primary: 35Q35; 35K55; 76S05; Secondary: 35K86; 47H05; 76D07}

\keywords{
Brinkman equations,
biofilm growth,
hard volume-fraction constraint,
evolution variational inequalities,
reaction--diffusion--advection systems,
monotone operators,
weak solutions,
porous media flow
}

\begin{document}

\begin{abstract}
We investigate a coupled system of partial differential equations modeling the
interaction between Brinkman flow, biofilm evolution, and nutrient transport in a
porous medium. The model captures the mutual influence between the fluid velocity
and the biofilm through drag and diffusion coefficients that depend on the local
biofilm volume fraction. A hard constraint on the admissible range of the biofilm
fraction is incorporated through the subdifferential of an indicator functional,
which leads naturally to an evolution variational inequality formulation for the
biofilm dynamics.

Assuming standard coercivity, ellipticity, and growth conditions on the model
coefficients and reaction terms, we prove the global-in-time existence of weak
solutions. The analysis relies on a decomposition of the system into three
interconnected subproblems: the Brinkman equation with a fixed biofilm profile,
the constrained biofilm evolution treated through maximal monotone operator theory,
and the nutrient equation viewed as a semilinear parabolic problem. These
components are then coupled through a Leray--Schauder type fixed-point argument,
with the passage to the limit justified by Aubin--Lions and Simon compactness
results.

We further establish the nonnegativity of the nutrient concentration under a
natural quasi-positivity assumption on the reaction term. Finally, we provide
conditional uniqueness results for weak solutions in two spatial dimensions under
additional smallness assumptions.
\end{abstract}

\maketitle

\section{Introduction}
\label{sec:intro}
Biofilms are ubiquitous in porous-media settings (soil, aquifers, filters, medical
devices) and interact strongly with the ambient flow and nutrient availability.
At the pore and Darcy scales, biofilm accumulation can reduce permeability, alter
drag, and modify transport pathways, while nutrient advection--diffusion and reaction
govern growth and decay. Mathematical models that couple fluid flow with biomass
growth and nutrient consumption therefore lead naturally to nonlinear PDE systems
with feedback through transport coefficients and reaction terms.

In this work we study a coupled flow--biofilm--nutrient system in a bounded porous
domain $\Omega\subset\mathbb{R}^d$ ($d=2,3$) proposed by \cite{shin2021}. The flow is described by the Brinkman
equations, which combine viscous diffusion with a Darcy-type resistance (drag) term.
The resistance (or inverse permeability) depends on the biofilm volume fraction $\phi$ and
thus encodes the reduction of pore space and mobility as biomass accumulates. Because of the limited space in porous media, biofilm growth is subject to a volume constraint, where the maximum density is denoted by $\phi^*$.
This modeling choice is particularly attractive in heterogeneous media and in
regimes where a sharp interface between Stokes and Darcy regions is not prescribed
a priori. Here we focus on the continuous-level analysis for a streamlined
Brinkman--biofilm--nutrient formulation, emphasizing existence (and, under
additional conditions, qualitative properties such as positivity and conditional
uniqueness).

Several continuum models for biofilm development have been proposed.
In the widely studied Eberl--Parker--van~Loosdrecht model \cite{EberlParkerLoosdrecht},
the biofilm volume fraction obeys a degenerate reaction--diffusion equation whose
diffusion coefficient vanishes at $\phi=0$ and blows up as $\phi\to\phi^\ast$.
The degeneracy ensures compact support and the singularity near the packing limit
enforces $\phi<\phi^\ast$, but the resulting PDE is highly nonlinear.
In the Klapper--Dockery framework \cite{KlapperDockery}, multiphase mixture theory is
used, leading to coupled systems with pressure-like terms.
In contrast, the present model enforces the constraint $0\le \phi\le\phi^\ast$
through the subdifferential of an indicator functional, yielding a variational inequality
formulation that avoids singular diffusivities while providing a rigorous constraint
mechanism; see also Wanner--Gujer \cite{WannerGujer} and van~Loosdrecht et
al.~\cite{Loosdrecht2002} for background on multispecies and structural biofilm models.

A central feature of biofilm growth in pores is the limited available space.
Consequently, the biomass fraction $\phi$ is subject to the pointwise constraint
$0\le \phi \le \phi^\ast$ (maximum packing/maximum density). This is not a mild
technicality: it changes the evolution into a constrained dynamics, and standard
parabolic theory does not apply directly. Following the variational approach used
in constrained diffusion and phase-fraction models, we enforce the constraint via
the subdifferential of the indicator functional of the closed convex set
\[
\mathcal{K}=\{v\in L^2(\Omega): 0\le v\le \phi^\ast\ \text{a.e.}\}.
\]
This leads to an evolution variational inequality (EVI) (equivalently, an evolution
inclusion governed by a maximal monotone operator), which is robust under weak
convergence and is well adapted to compactness arguments.

The coupled system consists of three interacting components: an elliptic Brinkman
problem for $(u,p)$, whose coefficients depend on the biofilm variable $\phi$; a
constrained parabolic evolution equation for the biofilm, involving nonlinear
diffusion and transport by the velocity field $u$; and a semilinear parabolic
nutrient equation with advection driven by the same velocity.

The main analytical challenges arise from the presence of the constraint
$0\leq \phi\leq \phi^\ast$, which leads to the multivalued term
$\partial I_{\mathcal{K}}(\phi)$, the nonlinear coupling between the unknowns
through the coefficients $\alpha(\phi)$, $D_1(\phi)$, $D_2(\phi)$ and the reaction
terms, as well as the treatment of the transport terms
$u\cdot\nabla\phi$ and $u\cdot\nabla\xi$. These advection terms are naturally
handled in a weak formulation through duality pairings, typically in
$H^{-1}(\Omega)$, at the level of weak solutions.

\subsection*{Main contribution.}
Under standard coercivity and growth assumptions on the drag coefficient and
diffusion tensors, together with suitable hypotheses on the reaction terms, we
establish the global-in-time existence of weak solutions. The proof relies on a
combination of the following arguments:
\begin{itemize}
\item for a fixed biofilm profile $\phi$, the Brinkman subproblem is solved by
exploiting the coercivity of the associated bilinear form and applying the
Lax--Milgram theorem in the divergence-free setting;
\item the constrained biofilm evolution is formulated as a parabolic inclusion
associated with the maximal monotone operator generated by the subdifferential of
$I_{\mathcal K}$. This provides an evolutionary variational inequality (EVI)
framework together with the required energy estimates;
\item the nutrient equation is treated as a semilinear parabolic problem with
uniformly elliptic diffusion and transport terms interpreted in a weak duality
framework;
\item the three components are then coupled through a compactness-based fixed-point
procedure of Schauder/Leray--Schauder type. The convergence of the nonlinear
couplings is obtained by means of Aubin--Lions/Simon compactness arguments.
\end{itemize}

Furthermore, we establish a weak maximum-principle property for the nutrient
concentration under a standard quasi-positivity assumption on the reaction terms.
We also discuss conditional uniqueness of weak solutions in two space dimensions,
under additional smallness conditions and suitable Lipschitz assumptions on the
nonlinearities.

The remainder of the article is organized as follows.
Section~\ref{sec:model} introduces the mathematical model and the structural
assumptions. Section~\ref{sec:weak} presents the weak formulation, emphasizing the
evolutionary variational inequality structure of the biofilm equation.
Section~\ref{sec:WP} is devoted to the main well-posedness result and the derivation
of the fundamental a priori estimates. The proof is developed in
Section~\ref{sec:proofWP} by analyzing the three coupled subproblems and applying a
fixed-point argument, followed by a discussion of positivity and uniqueness
properties. Finally, Section~\ref{sec:sim} provides numerical simulations
illustrating the influence of the fluid flow and biofilm permeability on the
system behavior.

\section{Model and assumptions}
\label{sec:model}
Let $\Om\subset\R^d$ ($d=2,3$) be bounded with Lipschitz boundary and $T>0$.
Unknowns are: velocity $u:(0,T)\times\Om\to\R^d$, pressure $p$, biofilm volume fraction
$\phi$, and nutrient concentration $\xi$.

\subsection*{Flow (Brinkman)}
For a given $\phi$ we consider
\begin{align}
- \mu \Delta u + \alpha(\phi)\,u + \nabla p &= f \quad \text{in }\Om, \label{eq:Brinkman}\\
\nabla\cdot u &= 0 \quad \text{in }\Om, \label{eq:div}\\
u&=0 \quad \text{on }\partial\Om. \label{eq:uBC}
\end{align}
Here $\mu>0$ and $f\in L^2(0,T;H^{-1}(\Om)^d)$.

\subsection*{Biofilm (constrained evolution)}
\begin{equation}\label{eq:bio}
\partial_t \phi - \nabla\cdot(D_1(\phi)\nabla\phi) + \partial I_{[0,\phi^\ast]}(\phi)
= R_1(\phi,\xi) - u\cdot\nabla\phi
\quad \text{in }(0,T)\times\Om,
\end{equation}
with homogeneous Neumann boundary condition $\nabla \phi \cdot n=0$ on $\partial\Om$ and
initial datum $\phi(0)=\phi_0$.
Note that $u\cdot\nabla\phi=\nabla\cdot\left(\phi u\right)$ since $\nabla\cdot u=0$.

\subsection*{Nutrient}
\begin{equation}\label{eq:nut}
\partial_t \xi - \nabla\cdot(D_2(\phi)\nabla\xi)
= R_2(\phi,\xi) - u\cdot\nabla\xi
\quad \text{in }(0,T)\times\Om,
\end{equation}
with no-flux (homogeneous Neumann) boundary condition $\nabla\xi\cdot n=0$ on
$\partial\Om$ and $\xi(0)=\xi_0$.
For definiteness we develop the analysis under this no-flux condition; an
inhomogeneous Dirichlet inlet $\xi=\xi_D$ on a portion of $\partial\Om$ with
$\xi_D\ge 0$ (as used in the simulations of Section~\ref{sec:sim}) is incorporated by
a standard lifting and leaves the existence and nonnegativity statements unchanged,
since the boundary contribution of the transport term still vanishes on the no-slip
walls and the negative part $\xi^-$ vanishes on the inlet.

\subsection*{Structural assumptions}
Fix $\phi^\ast>0$ and set $K:=[0,\phi^\ast]$.
\begin{itemize}
\item[(A1)] $\alpha:K\to\R$ is continuous and uniformly coercive:
$0<\alpha_0\le \alpha(s)\le \alpha_1<\infty$ for all $s\in K$.
\item[(A2)] $D_1,D_2:K\to\R$ are continuous and uniformly elliptic:
$0<d_i\le D_i(s)\le \bar d_i<\infty$ for all $s\in K$, $i=1,2$.
\item[(A3)] (Reaction growth) $R_1,R_2:K\times\R\to\R$ are Carath\'eodory and satisfy
for some $C>0$,
\[
|R_1(s,z)| + |R_2(s,z)| \le C\bigl(1+|z|\bigr)
\quad\text{for all }(s,z)\in K\times\R.
\]
\item[(A4)] (Local Lipschitz) For each $M>0$ there exists $L_M>0$ such that for all
$s\in K$ and $|z_1|,|z_2|\le M$,
\[
|R_1(s,z_1)-R_1(s,z_2)| + |R_2(s,z_1)-R_2(s,z_2)|
\le L_M |z_1-z_2|.
\]
\item[(A5)] Data: $f\in L^2(0,T;H^{-1}(\Om)^d)$, $\phi_0\in L^2(\Om)$ with $0\le \phi_0\le\phi^\ast$
a.e., and $\xi_0\in L^2(\Om)$ (optionally $\xi_0\ge0$ a.e.\ for positivity).
\item[(A6)] (Quasi-positivity for the nutrient) For all $s\in K$, one has
\[
R_2(s,0)\ge 0,
\]
and for every $M>0$ there exists $C_M>0$ such that for all $s\in K$ and
$z\in[-M,M]$,
\[
R_2(s,z)\ge -C_M\, z^-,
\qquad z^-:=\max\{-z,0\}.
\]
\item[(A7)] (Global Lipschitz for uniqueness) There exists $L>0$ such that for all
$s_1,s_2\in K$ and all $z_1,z_2\in\R$,
\[
|R_1(s_1,z_1)-R_1(s_2,z_2)| + |R_2(s_1,z_1)-R_2(s_2,z_2)|
\le L\bigl(|s_1-s_2|+|z_1-z_2|\bigr).
\]
\item[(A8)] (Lipschitz drag for uniqueness) There exists $L_\alpha>0$ such that for all
$s_1,s_2\in K$,
\[
|\alpha(s_1)-\alpha(s_2)|\le L_\alpha |s_1-s_2|.
\]
\item[(A9)] (Lipschitz diffusivities for uniqueness) There exists $L_D>0$ such that
for all $s_1,s_2\in K$,
\[
|D_1(s_1)-D_1(s_2)| + |D_2(s_1)-D_2(s_2)| \le L_D\,|s_1-s_2|.
\]
\end{itemize}

\begin{remark}[Representative coefficients]\label{rem:coeff-examples}\rm\, 
The analysis uses only the abstract hypotheses \textup{(A1)--(A9)}; explicit formulas
for $\alpha,D_1,D_2$ are needed solely for the simulations.
A representative drag satisfying \textup{(A1)} is
\[
\alpha(\phi)=\frac{\mu}{k_b+k_0\,(1-\phi/\phi^\ast)^2},
\qquad k_b,k_0>0,
\]
used in Section~\ref{sec:sim}, for which $\mu/(k_b+k_0)\le\alpha(\phi)\le\mu/k_b$.
For the biofilm diffusivity one may consider the formula in \cite{shin2021}, which is a bounded, non-degenerate regularization of
the Eberl--Parker--van~Loosdrecht diffusivity \cite{EberlParkerLoosdrecht} (which is
itself degenerate at $\phi=0$ and singular at $\phi=\phi^\ast$), and for the nutrient
a constant or smoothly $\phi$-dependent $D_2$; both satisfy \textup{(A2)}.
\end{remark}

\begin{remark}\rm\,
Assumption \textup{(A6)} is standard in reaction--diffusion systems and guarantees
nonnegativity preservation in a weak sense; see, e.g., \cite{Lady,Showalter}.
A typical example is
\[
R_2(\phi,\xi)=S(\phi)-c(\phi)\,\xi
\quad\text{with}\quad S(\phi)\ge 0,\; c(\phi)\ge 0.
\]
\end{remark}

\begin{remark}\label{rem:indicator}\rm\,
The constraint term $\partial I_{[0,\phi^\ast]}(\phi)$ is the subdifferential of the
indicator functional of the closed convex set $\mathcal{K}$ (defined in Section~\ref{sec:intro}),
which yields an evolution variational inequality formulation; see, e.g.,
\cite{Barbu,Brz,KS}.
Throughout we use $K=[0,\phi^\ast]\subset\R$ for the interval of admissible values
and $\mathcal{K}\subset L^2(\Om)$ for the corresponding $L^2$-constraint set.
\end{remark}

\section{Weak formulation}
\label{sec:weak}
Let
\[
V:=H^1(\Om), \qquad V':=H^{-1}(\Om), \qquad
V_\sigma:=\{v\in H_0^1(\Om)^d:\nabla\cdot v=0\}.
\]
We interpret transport terms in $V'$:
for $u\in L^2(0,T;H_0^1(\Om)^d)$ and $\psi\in V$,
\[
\langle u\cdot\nabla\phi,\psi\rangle
:=-\int_\Om \phi\, u\cdot\nabla\psi\,dx,
\qquad
\langle u\cdot\nabla\xi,\psi\rangle
:=-\int_\Om \xi\, u\cdot\nabla\psi\,dx.
\]
In particular, if $\phi,\xi\in L^2(\Om)$ and $u\in H_0^1(\Om)^d$, then
$u\cdot\nabla\phi$ and $u\cdot\nabla\xi$ are well-defined elements of $V'$.

\begin{definition}\label{def:weak}
A triple $(u,\phi,\xi)$ is a weak solution of \eqref{eq:Brinkman}--\eqref{eq:nut} if:
\begin{itemize}
\item $u\in L^2(0,T;V_\sigma)$ satisfies \eqref{eq:Brinkman}--\eqref{eq:div} in the
usual weak sense for a.e.\ $t$ with coefficient $\alpha(\phi(t))$;
\item $\phi\in L^2(0,T;V)\cap H^1(0,T;V')$ with $\phi(t)\in \mathcal{K}$ for a.e.\ $t$,
and for a.e.\ $t\in(0,T)$,
\begin{equation*}
\begin{aligned}
\langle \partial_t\phi(t),\, v-\phi(t)\rangle
+ \int_\Om D_1(\phi(t))\nabla\phi(t)\cdot\nabla\bigl(v-\phi(t)\bigr)\,dx
&\ge \int_\Om R_1(\phi(t),\xi(t))\bigl(v-\phi(t)\bigr)\,dx \\
&\quad + \langle u(t)\cdot\nabla\phi(t),\, \phi(t)-v\rangle,
\end{aligned}
\end{equation*}
for all $v\in\mathcal{K}$ \textup{(EVI)};
\item $\xi\in L^2(0,T;V)\cap H^1(0,T;V')$ and for all $\psi\in V$ and a.e.\ $t$,
\[
\langle \partial_t\xi(t),\psi\rangle
+ \int_\Om D_2(\phi(t))\nabla\xi(t)\cdot\nabla\psi\,dx
= \int_\Om R_2(\phi(t),\xi(t))\psi\,dx
+ \langle u(t)\cdot\nabla\xi(t),\psi\rangle;
\]
\item $\phi(0)=\phi_0$ and $\xi(0)=\xi_0$ in $L^2(\Om)$.
\end{itemize}
\end{definition}

\section{Main well-posedness result}
\label{sec:WP}
\begin{theorem}[Existence, nonnegativity of $\xi$, and conditional uniqueness]\label{thm:exist}
Assume \textup{(A1)--(A5)}.
Then for every $T>0$ there exists at least one weak solution $(u,\phi,\xi)$ in the
sense of Definition~\ref{def:weak} such that
\[
u\in L^2(0,T;V_\sigma),\qquad
\phi,\xi \in L^2(0,T;V)\cap H^1(0,T;V'),
\qquad 0\le \phi\le\phi^\ast \ \text{a.e. in }(0,T)\times\Om,
\]
and the a priori bounds
\[
\|u\|_{L^2(0,T;H^1)}^2 + \|\phi\|_{L^\infty(0,T;L^2)}^2 + \|\xi\|_{L^\infty(0,T;L^2)}^2
+ \|\nabla\phi\|_{L^2(0,T;L^2)}^2 + \|\nabla\xi\|_{L^2(0,T;L^2)}^2 \le C
\]
hold with $C=C(T,\|f\|_{L^2(0,T;V_\sigma')},\|\phi_0\|_{L^2},\|\xi_0\|_{L^2})$.

Moreover, if \textup{(A6)} holds and $\xi_0\ge 0$ a.e.\ in $\Om$, then
\[
\xi(t,x)\ge 0 \quad\text{for a.e. }(t,x)\in(0,T)\times\Om.
\]

Finally, assume in addition \textup{(A7)--(A9)}, that $f\in L^\infty(0,T;V_\sigma')$,
and let $d=2$.
There exists a constant
$\varepsilon=\varepsilon(\mu,\alpha_0,d_i,L,L_\alpha,L_D,\Om)>0$ such that
if either
\[
T\le \varepsilon
\qquad\text{or}\qquad
\|f\|_{L^2(0,T;V_\sigma')} \le \varepsilon,
\]
then the weak solution is unique on $(0,T)$.
\end{theorem}

\begin{remark}[Uniqueness]\rm\,
Uniqueness generally requires additional assumptions (e.g.\ Lipschitz dependence in
both variables, smallness of data, or stronger regularity of the transport velocity).
The Lipschitz condition \textup{(A9)} on $D_1,D_2$ is needed to control the cross-terms
$\int_\Om (D_i(\phi_1)-D_i(\phi_2))\nabla\phi_2\cdot\nabla(\phi_1-\phi_2)\,dx$ that arise when
subtracting the equations for two solutions; see the proof of
Proposition~\ref{prop:uniq}, where we abbreviate $\delta\phi:=\phi_1-\phi_2$.
Further refinements (e.g.\ unconditional uniqueness under stronger regularity)
are left for future work; see, for instance, the framework in
\cite{Showalter,Temam} for related coupled parabolic systems.
\end{remark}

\section{Proof of Theorem~\ref{thm:exist}}
\label{sec:proofWP}
This section is devoted to the proof of Theorem~\ref{thm:exist}. First, we state a weak maximum principle needed in the sequel.
\begin{lemma}[Weak maximum principle for the nutrient]\label{lem:xi-nonneg}
Assume \textup{(A1)--(A6)} and $\xi_0\ge 0$ a.e.\ in $\Om$.
Let $(u,\phi,\xi)$ be a weak solution in the sense of Definition~\ref{def:weak}.
Then $\xi\ge 0$ a.e.\ in $(0,T)\times\Om$.
\end{lemma}

\begin{proof}
Set $\xi^-:=\max\{-\xi,0\}\in L^2(0,T;V)$.
Taking $\psi=\xi^-(t)$ as a test function in the weak formulation of the nutrient
equation \eqref{eq:nut} (justified by standard truncation arguments; see below), we obtain
for a.e.\ $t\in(0,T)$,
\begin{align*}
\frac12\frac{d}{dt}\|\xi^-(t)\|_{L^2(\Om)}^2
+ \int_\Om D_2(\phi(t))|\nabla\xi^-(t)|^2\,dx
&= \int_\Om R_2(\phi(t),\xi(t))\,\xi^-(t)\,dx
+ \langle u(t)\cdot\nabla\xi(t),\xi^-(t)\rangle.
\end{align*}
The transport term vanishes: since $\nabla\cdot u(t)=0$ and $u(t)|_{\partial\Om}=0$, we obtain
\begin{eqnarray*}
\langle u(t)\cdot\nabla\xi(t),\xi^-(t)\rangle
&=& -\int_\Om \xi(t)\,u(t)\cdot\nabla\xi^-(t)\,dx\\
&=& \int_\Om \xi^-(t)\,u(t)\cdot\nabla\xi^-(t)\,dx\\
&=& \frac12\int_\Om u(t)\cdot\nabla|\xi^-(t)|^2\,dx\\
&=&0.    
\end{eqnarray*}

Using (A2) we have $D_2(\phi)\ge d_2>0$ and thus
\begin{equation}\label{eq:xi-neg-ineq}
\frac12\frac{d}{dt}\|\xi^-(t)\|_{L^2}^2
+ d_2\|\nabla\xi^-(t)\|_{L^2}^2
\le \int_\Om R_2(\phi(t),\xi(t))\,\xi^-(t)\,dx.
\end{equation}

We now detail the truncation argument for the right-hand side.
For $n\in\N$, define $\xi_n:=\max\{\min\{\xi,n\},-n\}$,
so that $|\xi_n|\le n$ a.e.\ and $\xi_n\to\xi$ in $L^2(0,T;V)$.
On the set $\{|\xi|\le n\}$ we have $|\xi|=|\xi_n|\le n$, so~(A6) gives
$R_2(\phi,\xi)\,\xi^- \le C_n\,|\xi^-|^2$.
On the complementary set $\{|\xi|>n\}$, the linear growth bound~(A3) gives
$|R_2(\phi,\xi)|\le C(1+|\xi|)$, whence
$|R_2(\phi,\xi)|\,|\xi^-|\le C(1+|\xi|)\,|\xi^-|\le C(1+|\xi|)^2$,
and the contribution from $\{|\xi|>n\}$ can be made arbitrarily small
for $n$ large by Chebyshev's inequality \footnote{In its simplest form, Chebyshev's inequality states that, for any $\lambda>0$,
\[
\left|\{x\in\Omega:\ |f(x)|\geq \lambda\}\right|
\leq
\frac{1}{\lambda}\int_{\Omega}|f(x)|\,dx .
\]
Here $|\cdot|$ denotes the Lebesgue measure.} and the a priori bound
$\xi\in L^\infty(0,T;L^2(\Om))$ (available from the existence energy estimates).
In total, for some constant $C>0$ depending on
$\|\xi\|_{L^\infty(0,T;L^2)}$,
\[
\int_\Om R_2(\phi(t),\xi(t))\,\xi^-(t)\,dx
\le C\,\|\xi^-(t)\|_{L^2}^2
\quad\text{for a.e.\ }t\in(0,T).
\]
Therefore, from \eqref{eq:xi-neg-ineq},
\[
\frac{d}{dt}\|\xi^-(t)\|_{L^2}^2 \le 2C\,\|\xi^-(t)\|_{L^2}^2.
\]
Since $\xi_0\ge 0$, we have $\xi^-(0)=0$, and Gr\"onwall's inequality \cite[p. 664]{Evans} gives
$\xi^-(t)\equiv 0$ for all $t\in[0,T]$. Hence $\xi\ge 0$ a.e.
\end{proof}

We now give the main steps of the proof, emphasizing the time-parametrized nature of the
Brinkman subproblem and the well-definedness of the transport terms in $V'$.

\subsection{Step 1: Brinkman subproblem (frozen biofilm)}
Fix $\bar\phi\in L^\infty(\Om)$ with $0\le \bar\phi\le\phi^\ast$.
By (A1), the bilinear form on $V_\sigma$,
\[
a_{\bar\phi}(u,v):=\mu\int_\Om \nabla u:\nabla v\,dx + \int_\Om \alpha(\bar\phi)\,u\cdot v\,dx,
\]
is continuous and coercive. Hence, by Lax--Milgram theorem~\cite[Corollary 5.8, p. 140]{Brezis}, there exists a unique
$u=S_{\mathrm{flow}}(\bar\phi)\in V_\sigma$ solving \eqref{eq:Brinkman}--\eqref{eq:div} with
$\alpha(\bar\phi)$ and satisfying
\begin{equation}\label{eq:u_bound}
\|u\|_{H^1(\Om)} \le C \|f(t)\|_{V_\sigma'}.
\end{equation}
In the coupled problem, this construction is applied for a.e.\ $t\in(0,T)$ with
$\bar\phi=\bar\phi(t)$ and forcing $f(t)$,
which yields $u\in L^2(0,T;V_\sigma)$ and the time-integrated estimate
$\|u\|_{L^2(0,T;H^1)} \le C\|f\|_{L^2(0,T;V_\sigma')}$.

\subsection{Step 2: Constrained biofilm evolution for given \texorpdfstring{$(\bar u,\bar\xi,\bar\phi)$}{(ubar, xibar, phibar)}}
Fix $\bar u\in L^2(0,T;V_\sigma)$, $\bar\xi\in L^2(0,T;L^2(\Om))$, and
$\bar\phi\in L^\infty((0,T)\times\Om)$ with values in $K$.
Define
\[
\mathcal{A}_{\bar\phi}(w):=-\nabla\cdot(D_1(\bar\phi)\nabla w)
\quad \text{with Neumann b.c.}
\]
and the maximal monotone operator $\partial I_{\mathcal{K}}$ in $L^2(\Om)$.
Then the biofilm subproblem (with all nonlinear coefficients and reaction terms
evaluated at the frozen data $(\bar\phi,\bar\xi)$) reads in $V'$ as the
evolution inclusion
\begin{equation}\label{eq:evol_incl}
\partial_t\phi + \mathcal{A}_{\bar\phi}(\phi) + \bar u\cdot\nabla\phi
+ \partial I_{\mathcal{K}}(\phi)
\ni R_1(\bar\phi,\bar\xi).
\end{equation}
The right-hand side $R_1(\bar\phi,\bar\xi)$ belongs to $L^2(0,T;L^2)$
by (A3) since $\bar\phi\in K$ a.e.\ and $\bar\xi\in L^2$, and is
independent of the unknown $\phi$.
The operator $\mathcal{A}_{\bar\phi}+\bar u\cdot\nabla$ on the left-hand side is
\emph{linear} in $\phi$: the diffusion part $\mathcal{A}_{\bar\phi}$ is coercive
with constant $d_1>0$ by (A2), and the transport part $\bar u\cdot\nabla\phi$
is skew-symmetric in the $V'$--$V$ duality since
$\langle \bar u\cdot\nabla\phi,\phi\rangle=0$ (cf.\ Section~3).
Therefore, the sum $\mathcal{A}_{\bar\phi}+\bar u\cdot\nabla$ is a bounded linear
coercive operator from $V$ to $V'$, and $\partial I_{\mathcal{K}}$ is maximal monotone.
By standard theory for evolution inclusions governed by sums of linear coercive
and maximal monotone operators
(see \cite[Ch.~III]{Brz}, \cite[Ch.~IV]{Barbu}, and \cite[Ch.~V]{Showalter}),
there exists a unique solution
$\phi\in L^2(0,T;V)\cap H^1(0,T;V')$ with $\phi(t)\in\mathcal{K}$ a.e.

\begin{remark}\rm\,
At the fixed point (Step~4 below), $\bar\phi=\phi$ and $\bar\xi=\xi$, so
\eqref{eq:evol_incl} reduces to the original biofilm equation \eqref{eq:bio}.
Freezing the reaction term $R_1$ at $(\bar\phi,\bar\xi)$
rather than evaluating it at the unknown $\phi$ ensures that the subproblem
is a standard evolution inclusion with given data and avoids
the additional difficulty of a non-monotone $\phi$-dependent perturbation on
the right-hand side.
\end{remark}

\subsection{Step 3: Nutrient equation for given $(\bar u,\bar\phi)$}
Given $\bar u\in L^2(0,T;V_\sigma)$ and $\bar\phi\in L^\infty((0,T)\times\Om)$ with values in $K$,
the nutrient equation \eqref{eq:nut} is a semilinear parabolic equation with uniformly elliptic
diffusion $D_2(\bar\phi)$ and transport in $V'$.
Using standard Galerkin or monotonicity methods (cf.\ \cite{Lady,Temam,Showalter}),
one obtains a solution $\xi\in L^2(0,T;V)\cap H^1(0,T;V')$.

\subsection{Step 4: Fixed point and compactness}
The a priori energy estimates (testing the EVI for $\phi$ with $v=0$ and $v=\phi^\ast$,
and the nutrient equation with $\psi=\xi$, combined with
the flow bound \eqref{eq:u_bound}) give, for any solution $(\phi,\xi)$ of the decoupled
system in Steps~1--3 with input $(\bar\phi,\bar\xi)$,
\begin{equation}\label{eq:apriori_ball}
\|\phi\|_{L^2(0,T;V)\cap H^1(0,T;V')}
+ \|\xi\|_{L^2(0,T;V)\cap H^1(0,T;V')} \le R,
\end{equation}
where $R=R(T,\|f\|_{L^2(0,T;V_\sigma')},\|\phi_0\|_{L^2},\|\xi_0\|_{L^2})>0$
is independent of $(\bar\phi,\bar\xi)$.
Indeed, the constraint $\phi\in\mathcal{K}$ provides $L^\infty$-control on $\phi$,
and (A3) controls the reaction terms linearly in $\|\xi\|_{L^2}$, from which
a Gr\"onwall argument closes the bound.

Define the closed bounded convex set
\[
\mathcal{X}_R:=\Bigl\{(\phi,\xi)\in L^2(0,T;L^2(\Om))^2:\
\phi(t)\in\mathcal{K}\ \text{a.e.},\
\|\phi\|_{L^2(0,T;V)\cap H^1(0,T;V')}
+ \|\xi\|_{L^2(0,T;V)\cap H^1(0,T;V')} \le R
\Bigr\},
\]
equipped with the $L^2(0,T;L^2(\Om))^2$ topology, and the map
$\mathcal{T}:\mathcal{X}_R\to\mathcal{X}_R$ by
\[
\mathcal{T}(\bar\phi,\bar\xi):=(\phi,\xi),
\]
where $u=S_{\mathrm{flow}}(\bar\phi(t))$ for a.e.\ $t$, then $\phi$ solves
\eqref{eq:evol_incl}, and $\xi$ solves \eqref{eq:nut} with $(u,\bar\phi)$.
By \eqref{eq:apriori_ball}, $\mathcal{T}$ maps $\mathcal{X}_R$ into itself.

By Aubin--Lions and Simon's compactness criterion \cite{Simon},
bounded sets in $L^2(0,T;V)\cap H^1(0,T;V')$ are relatively compact in
$L^2(0,T;L^2(\Om))$; hence $\mathcal{T}$ is compact.

\begin{lemma}[Continuity of the fixed-point map]\label{lem:Tcont}
The map $\mathcal{T}:\mathcal{X}_R\to\mathcal{X}_R$ is continuous in the
$L^2(0,T;L^2(\Om))^2$ topology.
\end{lemma}

\begin{proof}
Let $(\bar\phi_n,\bar\xi_n)\to(\bar\phi,\bar\xi)$ in $L^2(0,T;L^2(\Om))^2$ with
$(\bar\phi_n,\bar\xi_n)\in\mathcal{X}_R$ and $\bar\phi_n,\bar\phi\in\mathcal{K}$ a.e.
Set $(\phi_n,\xi_n):=\mathcal{T}(\bar\phi_n,\bar\xi_n)$.

By the uniform bound \eqref{eq:apriori_ball} and Aubin--Lions/Simon compactness,
$\{(\phi_n,\xi_n)\}$ is relatively compact in $L^2(0,T;L^2)^2$.
Let $(\phi,\xi)$ be any subsequential limit.
Passing to the limit in each of the three subproblems
(using stability of $S_{\mathrm{flow}}$ under a.e.-convergence of $\bar\phi_n$,
continuous dependence of the evolution inclusion on data in $L^2(0,T;V')$,
and standard parabolic stability for the nutrient equation),
one verifies that $(\phi,\xi)=\mathcal{T}(\bar\phi,\bar\xi)$.
Since the limit $\mathcal{T}(\bar\phi,\bar\xi)$ is uniquely determined
(each subproblem in Steps~1--3 has a unique solution for given data),
every convergent subsequence of $\{(\phi_n,\xi_n)\}$ has the same limit.
A standard argument then implies that the full sequence converges:
$\mathcal{T}(\bar\phi_n,\bar\xi_n)\to\mathcal{T}(\bar\phi,\bar\xi)$ in
$L^2(0,T;L^2)^2$.
\end{proof}

Therefore, by Schauder's fixed point theorem
(applied to the continuous compact map $\mathcal{T}$ on the closed bounded convex set
$\mathcal{X}_R$ in the Banach space $L^2(0,T;L^2(\Om))^2$),
$\mathcal{T}$ admits a fixed point $(\phi,\xi)\in\mathcal{X}_R$.
Setting $u=S_{\mathrm{flow}}(\phi(t))$ for a.e.\ $t$ yields a weak solution $(u,\phi,\xi)$ in the sense
of Definition~\ref{def:weak}. The global bounds follow from the uniform energy estimates. \qed

\begin{proposition}[Conditional uniqueness in $2D$]\label{prop:uniq}
Assume \textup{(A1)--(A5)} and \textup{(A7)--(A9)}, that $f\in L^\infty(0,T;V_\sigma')$,
and let $d=2$.
Let $(u_i,\phi_i,\xi_i)$, $i=1,2$, be two weak solutions with the same initial data.
There exists $\varepsilon>0$ depending only on the structural constants and $\Om$
such that if $T\le \varepsilon$ or $\|f\|_{L^2(0,T;V_\sigma')}\le \varepsilon$, then
\[
u_1=u_2,\qquad \phi_1=\phi_2,\qquad \xi_1=\xi_2 \quad \text{a.e. in }(0,T)\times\Om.
\]
\end{proposition}

\begin{proof}
Set $\delta u=u_1-u_2$, $\delta\phi=\phi_1-\phi_2$, $\delta\xi=\xi_1-\xi_2$.
We carry out the energy method in detail.
The standing assumption $f\in L^\infty(0,T;V_\sigma')$ guarantees, via
\eqref{eq:u_bound}, that $\|u_i(t)\|_{H^1}\le C\|f\|_{L^\infty(0,T;V_\sigma')}$ for
a.e.\ $t$ and $i=1,2$; this bound is used in step~(i) below.

\smallskip
\noindent\emph{(i) Flow stability.}
Subtracting the Brinkman problems and testing with $\delta u\in V_\sigma$, using coercivity of $\alpha(\cdot)$,
we obtain for a.e.\ $t$,
\[
\mu\|\nabla\delta u(t)\|_{L^2}^2 + \alpha_0\|\delta u(t)\|_{L^2}^2
\le \int_\Om |\alpha(\phi_1(t))-\alpha(\phi_2(t))|\,|u_2(t)|\,|\delta u(t)|\,dx
\le L_\alpha\int_\Om |\delta\phi|\,|u_2|\,|\delta u|\,dx,
\]
where we used (A8). In dimension $d=2$ the velocity $u_2$ need not belong to
$L^\infty(\Om)$: under the standing assumption $f\in L^\infty(0,T;V_\sigma')$ the
Brinkman solution of Step~1 satisfies only $u_2(t)\in V_\sigma\subset H_0^1(\Om)^2$,
and in two dimensions the borderline Sobolev embedding $H^1(\Om)\hookrightarrow
L^q(\Om)$ holds for every $q\in[1,\infty)$ but fails for $q=\infty$ (see, e.g.,
\cite[Ch.~9]{Brezis}; a standard counterexample on a ball is
$x\mapsto(-\log|x|)^\beta$ with $0<\beta<1/2$). The cubic term must therefore be
estimated by interpolation rather than by $\|u_2\|_{L^\infty}$.
By H\"older's inequality with exponents $(4,4,2)$, the
Ladyzhenskaya inequality $\|\delta\phi\|_{L^4}\le C\|\delta\phi\|_{L^2}^{1/2}\|\delta\phi\|_{H^1}^{1/2}$,
the embedding $\|u_2\|_{L^4}\le C\|u_2\|_{H^1}$, and Young's inequality, we obtain
for every $\sigma>0$,
\begin{equation}\label{eq:uniq-flow}
\mu\|\nabla\delta u(t)\|_{L^2}^2 + \frac{\alpha_0}{2}\|\delta u(t)\|_{L^2}^2
\le \sigma\,\|\nabla\delta\phi(t)\|_{L^2}^2 + C\,\|\delta\phi(t)\|_{L^2}^2,
\end{equation}
where $C=C(\sigma,\mu,\alpha_0,L_\alpha,\Om,\|f\|_{L^\infty(0,T;V_\sigma')})$.
(The gradient term $\sigma\|\nabla\delta\phi\|_{L^2}^2$ is harmless: it will be absorbed
into the biofilm diffusion in step~(iv).)

\smallskip
\noindent\emph{(ii) Biofilm estimate (monotonicity).}
Using the EVI formulation for $\phi_i$ (testing the inequality for $\phi_1$ with
$v=\phi_2$ and vice versa, then adding) and monotonicity of $\partial I_{\mathcal K}$,
one derives
\begin{align*}
\frac12\frac{d}{dt}\|\delta\phi(t)\|_{L^2}^2
&+ \int_\Om D_1(\phi_1)\,|\nabla\delta\phi|^2\,dx \\
&\le -\int_\Om \bigl(D_1(\phi_1)-D_1(\phi_2)\bigr)\nabla\phi_2\cdot\nabla\delta\phi\,dx \\
&\quad + \int_\Om \bigl(R_1(\phi_1,\xi_1)-R_1(\phi_2,\xi_2)\bigr)\delta\phi\,dx
+ \int_\Om (\delta u\cdot\nabla\phi_1)\,\delta\phi\,dx.
\end{align*}
Here we have used the decomposition
$u_1\cdot\nabla\phi_1-u_2\cdot\nabla\phi_2=\delta u\cdot\nabla\phi_1+u_2\cdot\nabla\delta\phi$;
the contribution of the second term vanishes when tested against $\delta\phi$, since
\[
\int_\Om (u_2\cdot\nabla\delta\phi)\,\delta\phi\,dx
=\frac12\int_\Om u_2\cdot\nabla|\delta\phi|^2\,dx=0
\]
because $\nabla\cdot u_2=0$ and $u_2|_{\partial\Om}=0$ (cf.\ Lemma~\ref{lem:xi-nonneg}).
Hence only the term $\int_\Om (\delta u\cdot\nabla\phi_1)\,\delta\phi\,dx$ survives on
the right-hand side.

The \emph{diffusion cross-term} is estimated using (A9):
\[
\left|\int_\Om (D_1(\phi_1)-D_1(\phi_2))\nabla\phi_2\cdot\nabla\delta\phi\,dx\right|
\le L_D\int_\Om |\delta\phi|\,|\nabla\phi_2|\,|\nabla\delta\phi|\,dx.
\]
In $2D$, by Ladyzhenskaya's inequality
$\|w\|_{L^4}\le C\|w\|_{L^2}^{1/2}\|\nabla w\|_{L^2}^{1/2}$
and Young's inequality with parameter $\eta>0$,
\[
L_D\int_\Om |\delta\phi|\,|\nabla\phi_2|\,|\nabla\delta\phi|\,dx
\le \eta\|\nabla\delta\phi\|_{L^2}^2
+ C_\eta\|\nabla\phi_2\|_{L^2}^2\|\delta\phi\|_{L^2}^2.
\]
The \emph{reaction cross-term} is estimated by (A7):
$|R_1(\phi_1,\xi_1)-R_1(\phi_2,\xi_2)|\le L(|\delta\phi|+|\delta\xi|)$, giving
\[
\int_\Om (R_1(\phi_1,\xi_1)-R_1(\phi_2,\xi_2))\,\delta\phi\,dx
\le C\bigl(\|\delta\phi\|_{L^2}^2+\|\delta\xi\|_{L^2}^2\bigr).
\]
The \emph{transport cross-term} is handled by Ladyzhenskaya and Young:
\[
\left|\int_\Om (\delta u\cdot\nabla\phi_1)\,\delta\phi\,dx\right|
\le \eta \|\nabla\delta\phi\|_{L^2}^2 + \eta \|\nabla\delta u\|_{L^2}^2
+ C_\eta \|\nabla\phi_1\|_{L^2}^2 \|\delta\phi\|_{L^2}^2.
\]
Combining and using (A2) ($D_1(\phi_1)\ge d_1$), we obtain
\begin{equation}\label{eq:uniq-bio}
\frac12\frac{d}{dt}\|\delta\phi\|_{L^2}^2
+ (d_1-2\eta)\|\nabla\delta\phi\|_{L^2}^2
\le \eta\|\nabla\delta u\|_{L^2}^2
+ C\bigl(\|\nabla\phi_1\|_{L^2}^2+\|\nabla\phi_2\|_{L^2}^2+1\bigr)
\bigl(\|\delta\phi\|_{L^2}^2 + \|\delta\xi\|_{L^2}^2\bigr).
\end{equation}

\smallskip
\noindent\emph{(iii) Nutrient estimate.}
Subtracting the nutrient equations and testing with $\delta\xi$, one obtains analogously
(using (A7) for the reaction term and (A9) for the diffusion cross-term
$(D_2(\phi_1)-D_2(\phi_2))\nabla\xi_2\cdot\nabla\delta\xi$):
\begin{equation}\label{eq:uniq-nut}
\frac12\frac{d}{dt}\|\delta\xi\|_{L^2}^2 + (d_2-2\eta)\|\nabla\delta\xi\|_{L^2}^2
\le \eta\|\nabla\delta u\|_{L^2}^2
+ C\bigl(\|\nabla\xi_1\|_{L^2}^2 + \|\nabla\xi_2\|_{L^2}^2+1\bigr)
\bigl(\|\delta\phi\|_{L^2}^2 + \|\delta\xi\|_{L^2}^2\bigr).
\end{equation}

\smallskip
\noindent\emph{(iv) Gr\"onwall.}
Adding \eqref{eq:uniq-bio} and \eqref{eq:uniq-nut}, we control the two
$\eta\|\nabla\delta u\|_{L^2}^2$ terms on their right-hand sides by means of
\eqref{eq:uniq-flow}, which gives
\[
\eta\|\nabla\delta u\|_{L^2}^2
\le \frac{\eta}{\mu}\Bigl(\sigma\|\nabla\delta\phi\|_{L^2}^2 + C\|\delta\phi\|_{L^2}^2\Bigr).
\]
Choosing first $\sigma>0$ and then $\eta>0$ small enough that
$d_1-2\eta-2\eta\sigma/\mu>0$ and $d_2-2\eta>0$, all gradient terms on the right
are absorbed into the left-hand diffusion terms. This yields
\[
\frac{d}{dt}\Bigl(\|\delta\phi(t)\|_{L^2}^2+\|\delta\xi(t)\|_{L^2}^2\Bigr)
\le C(t)\Bigl(\|\delta\phi(t)\|_{L^2}^2+\|\delta\xi(t)\|_{L^2}^2\Bigr),
\]
where
\[
C(t)= C_0\bigl(1+\|\nabla\phi_1(t)\|_{L^2}^2+\|\nabla\phi_2(t)\|_{L^2}^2
+\|\nabla\xi_1(t)\|_{L^2}^2+\|\nabla\xi_2(t)\|_{L^2}^2\bigr)
\]
and $C_0$ depends on $\mu,\alpha_0,d_i,L,L_\alpha,L_D,\Om$ and on
$\|f\|_{L^\infty(0,T;V_\sigma')}$.
By the a priori estimates of Theorem~\ref{thm:exist},
$\int_0^T C(t)\,dt\le C_0\bigl(T+C_{\mathrm{apriori}}\bigr)<\infty$,
where $C_{\mathrm{apriori}}=\int_0^T\bigl(\|\nabla\phi_1\|_{L^2}^2+\|\nabla\phi_2\|_{L^2}^2
+\|\nabla\xi_1\|_{L^2}^2+\|\nabla\xi_2\|_{L^2}^2\bigr)\,dt$.
Since $\delta\phi(0)=\delta\xi(0)=0$, Gr\"onwall's inequality forces
$\delta\phi\equiv\delta\xi\equiv 0$ on $(0,T)$; in particular this holds under either
smallness condition $T\le\varepsilon$ or $\|f\|_{L^2(0,T;V_\sigma')}\le\varepsilon$.
Finally, \eqref{eq:uniq-flow} yields $\delta u\equiv 0$.
\end{proof}

\begin{remark}[Removability of the smallness condition]\label{rem:uniq-uncond}\rm\,
The argument above in fact establishes uniqueness \emph{without} any smallness
restriction. Indeed, under the standing assumption $f\in L^\infty(0,T;V_\sigma')$ the
coefficient $C(t)$ is integrable on $(0,T)$ by the a priori estimates of
Theorem~\ref{thm:exist}, so Gr\"onwall's inequality applied to
$E(t)=\|\delta\phi(t)\|_{L^2}^2+\|\delta\xi(t)\|_{L^2}^2$ with $E(0)=0$ forces
$E\equiv 0$ for every finite $T$. The smallness conditions retained in
Theorem~\ref{thm:exist} and Proposition~\ref{prop:uniq} are therefore sufficient but
not necessary; we keep them only to make the explicit dependence on the data
transparent. (For merely $L^2$-in-time forcing the velocity is only in
$L^2(0,T;H^1)$, in which case integrability of $C(t)$ is no longer automatic and a
smallness or higher-integrability hypothesis is genuinely required.)
\end{remark}

\section{Simulations}\label{sec:sim}

We present numerical simulations to illustrate the evolution of the coupled
Brinkman--biofilm--nutrient system and to complement the theoretical analysis.
The computational domain $\Om\subset\R^2$ is the unit square $(0,1)\times(0,1)$
(in $\mathrm{mm}^2$) containing circular obstacles representing a heterogeneous porous
medium, as depicted in Figure~\ref{fig:porous med}.
The biomass is initialized adhering to the solid obstacles with volume fraction
$\phi_{0}=0.7$, with no nutrient present initially; the clean medium is assigned base
permeability $k=10^{-5}$.
We set homogeneous Neumann boundary conditions for the biomass at all walls.
Similarly, homogeneous Neumann boundary conditions are set for the nutrient at all
walls except the left wall, where it is fed by a constant nutrient supply $\xi_D=1$.
The fluid flows from left to right with initial parabolic velocity and with no-slip
conditions at the top and bottom walls.
We take the maximum biomass density to be $\phi^\ast=1$, and the biomass is regarded
as mature once its density reaches the maturity threshold $\phi_s=0.9$.

\noindent\textbf{Numerical method.}  We perform simulations using the BIO2020 MATLAB code which is available on the GitHub platform \cite{Bio2020}. The domain $\Om$ is discretized uniformly into rectangles of area $\Delta x \times \Delta y$. The solution of the system \eqref{eq:Brinkman}--\eqref{eq:nut} is obtained by implementing the operator splitting method \cite{LeVeque02}, where for each time step, the fluid velocity $u$ is obtained from the Brinkman equations \eqref{eq:Brinkman} then it is used to find the advection parts of the biofilm--nutrient system \eqref{eq:bio}--\eqref{eq:nut}. Then the obtained solutions are used to find the diffusion--reaction parts in the system \eqref{eq:bio}--\eqref{eq:nut}. The biofilm--nutrient system \eqref{eq:bio}--\eqref{eq:nut} is approximated in time using the implicit backward Euler method. The advection parts are approximated using the upwind scheme \cite{LeVeque02} while the diffusion--reaction parts are approximated using the cell-centered finite difference method \cite{RussellWheeler83}.   The Brinkman system \eqref{eq:Brinkman} is approximated using the marker--and--cell method \cite{HarlowWelch} (see also \cite{Patankar80}). To deal with the volume--fraction constraint $0\le \phi\le \phi^\ast$, Lagrange multiplier and semi-smooth Newton methods \cite{Ulbrich2011} are implemented.

\noindent\textbf{Drag coefficient.}
The drag function $\alpha(\phi)$ in the Brinkman equations encodes the reduction
of permeability due to biofilm growth. We use 
\[
\alpha(\phi) = \frac{\mu}{k_b + k_0\,(1-\phi/\phi^\ast)^2},
\]
where $k_0>0$ is the base permeability of the clean porous medium and
$k_b>0$ is the intrinsic biofilm permeability.
This function satisfies assumption~(A1) for any $k_b>0$, since
$\mu/(k_b+k_0)\le\alpha(\phi)\le\mu/k_b$.
Smaller values of $k_b$ correspond to less permeable (denser) biofilm,
leading to stronger flow--biofilm coupling. Such Kozeny--Carman-type clogging
closures, in which the drag (inverse permeability) increases as the available pore
space decreases, are standard in porous-media modeling; see, e.g.,
\cite{Bear,NieldBejan,Vafai}.

Our study evaluates two distinct cases.

\noindent\textbf{Case 1} investigates the effect of fluid flow on biofilm growth
by comparing two scenarios: a \emph{static scenario}, where the medium is
nutrient-rich but the fluid remains at rest, and a \emph{dynamic scenario},
where the fluid flows from left to right, with nutrients continuously
and intensively injected through the inlet; see Figure~\ref{fig:comparison}.
In the static case, biofilm growth is limited by nutrient depletion near the
colony centers, while in the dynamic case, the sustained nutrient supply
transported from the inlet leads to more extensive and asymmetric biofilm
development.

\noindent\textbf{Case 2} investigates the effect of biofilm permeability $k_b$
on the flow and on the biofilm growth;
see Figures~\ref{fig:placeholder} and~\ref{fig:case 3}.
When $k_b$ is very small ($k_b=10^{-15}$), the biofilm acts almost as a solid
obstacle, strongly redirecting the flow and limiting nutrient penetration into
the biofilm interior.
At intermediate permeability ($k_b=10^{-5}$), some flow penetrates the biofilm
region, providing nutrients and promoting growth.

\noindent\textbf{Nonnegativity of the nutrient.}
Figure~\ref{fig:permeability} reports the minimum nutrient concentration over time
and confirms that it remains nonnegative throughout the simulation, in agreement with
Lemma~\ref{lem:xi-nonneg}. The curves for $k_b=10^{-15}$ and $k_b=10^{-5}$ are nearly
indistinguishable, indicating that the minimum nutrient level is essentially
insensitive to the biofilm permeability in this range.

\begin{figure}[H]
    \centering
    \includegraphics[width=0.4\linewidth]{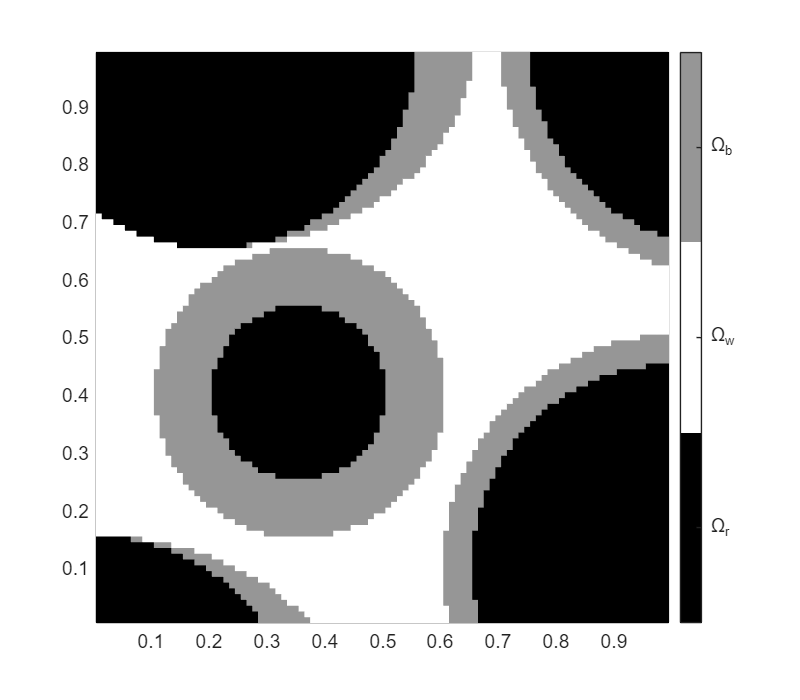}
    \includegraphics[width=0.4\linewidth]{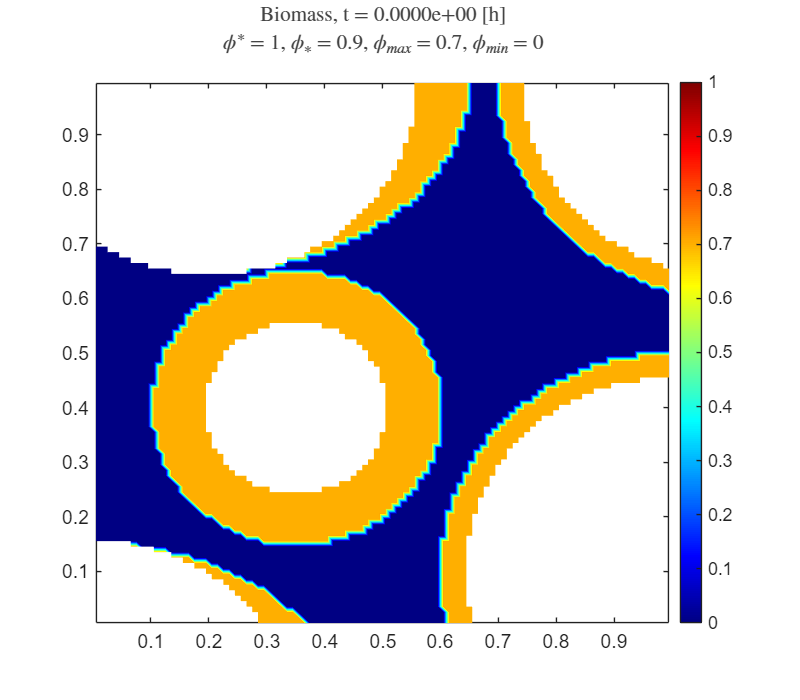}
   \includegraphics[width=0.4\linewidth]{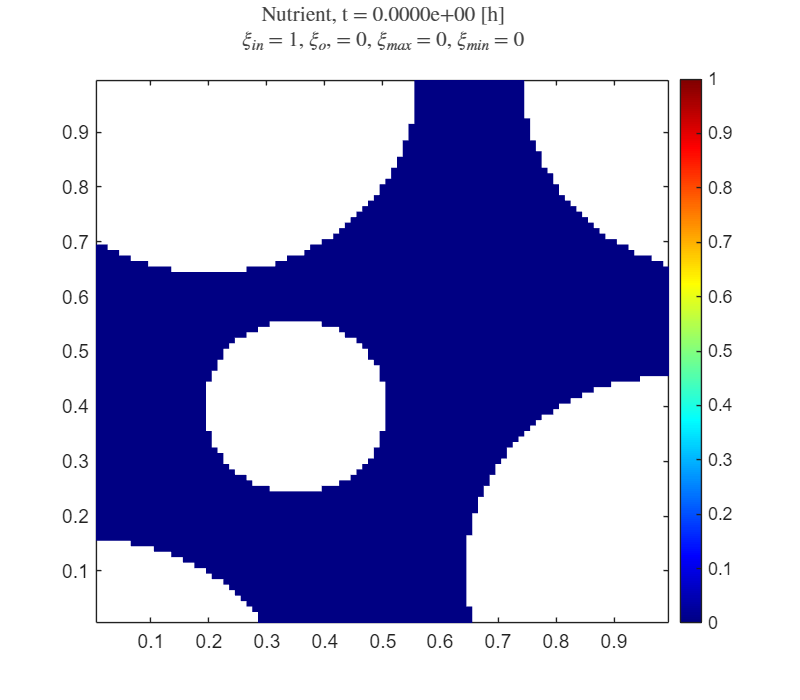}
   \includegraphics[width=0.4\linewidth]{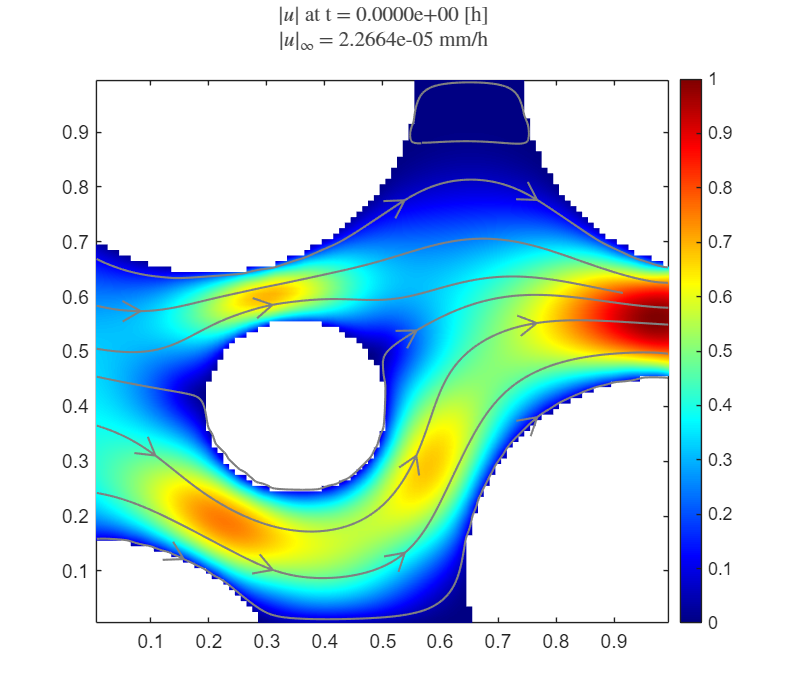}
    \caption{Initial porous medium $\Omega=\Omega_r\cup \Omega_b \cup \Omega_n$; $\Omega_r$: rock domain, $\Omega_b$: biomass domain, $\Omega_n:$ void space }
    \label{fig:porous med}
\end{figure}
\begin{figure}[H]
    \centering
      \includegraphics[width=0.34\linewidth]{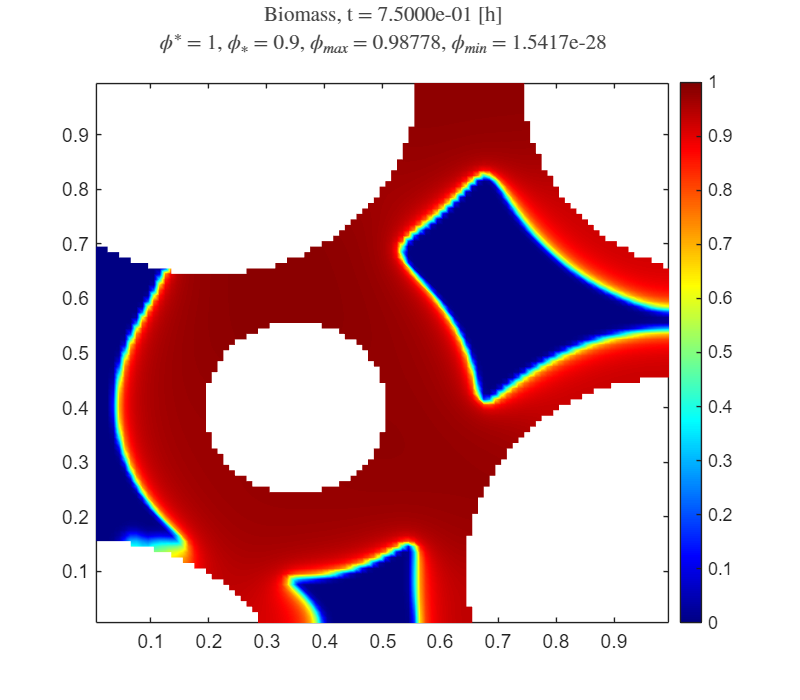}
    \includegraphics[width=0.34\linewidth]{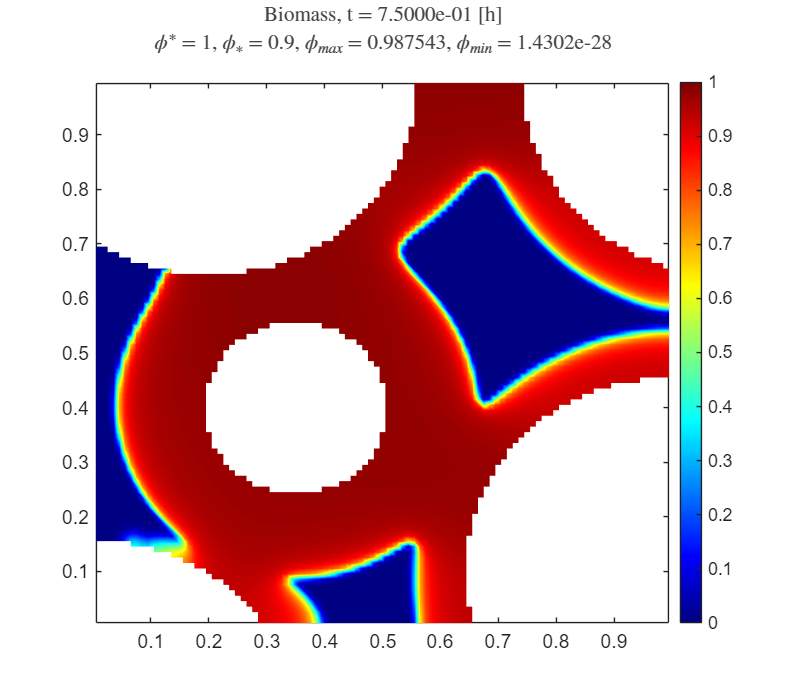}
     \includegraphics[width=0.34\linewidth]{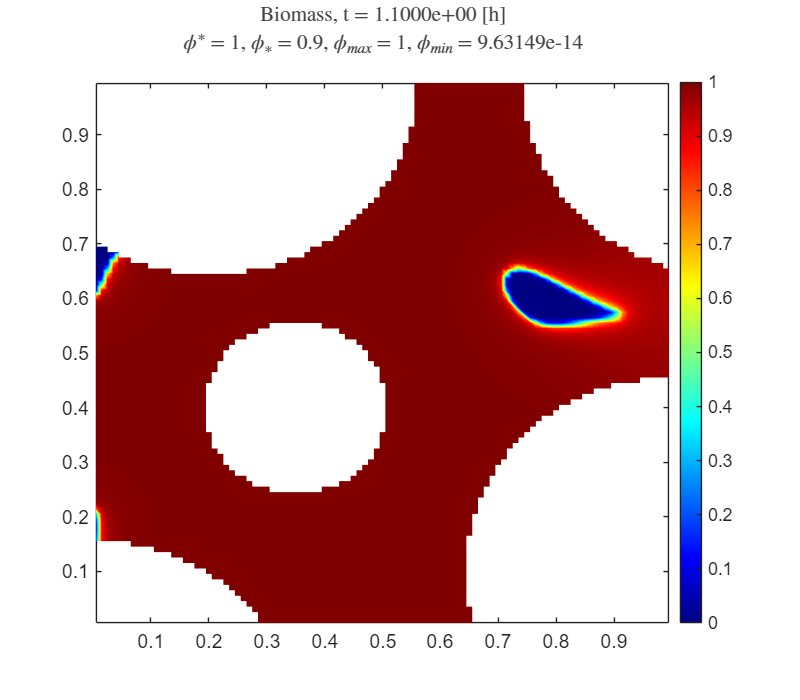}
    \includegraphics[width=0.34\linewidth]{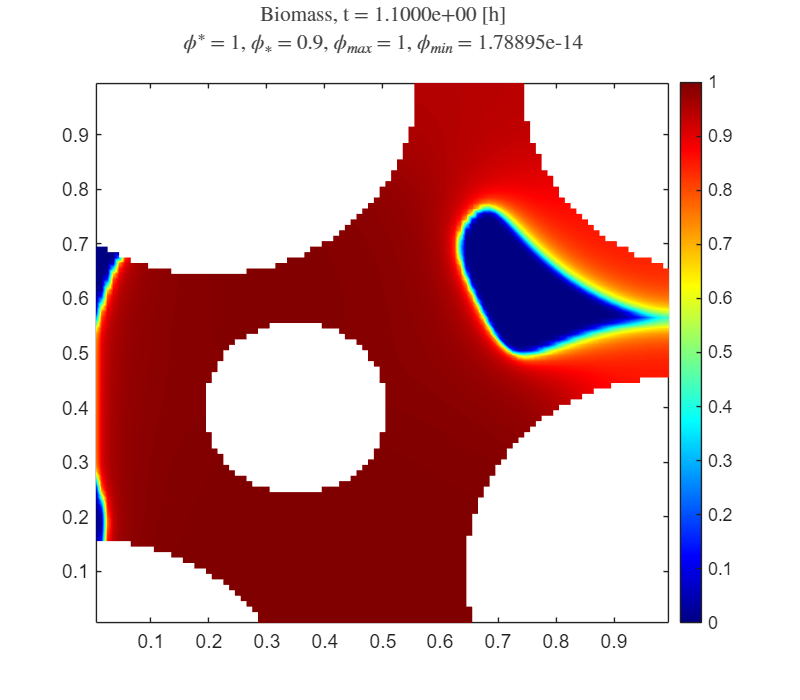}
     \includegraphics[width=0.34\linewidth]{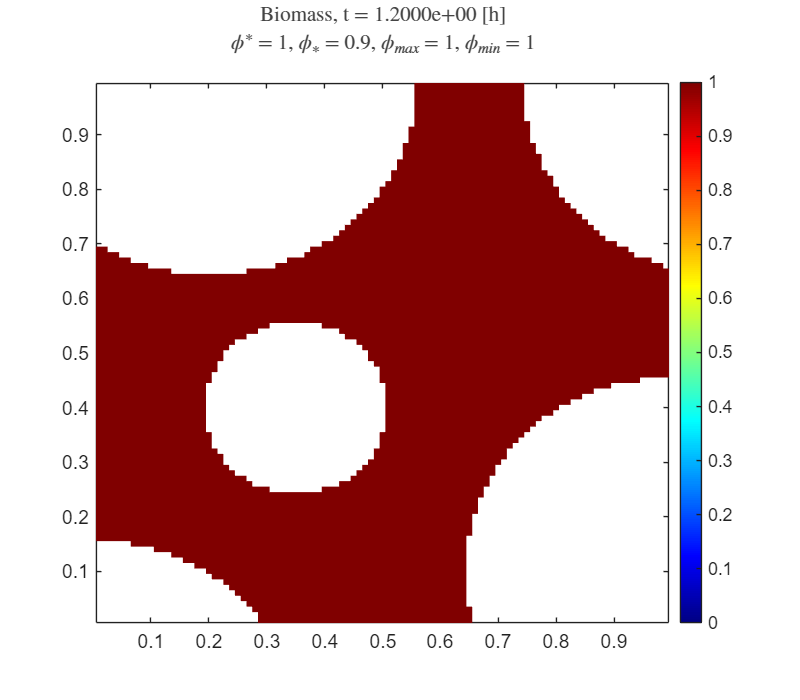}
    \includegraphics[width=0.34\linewidth]{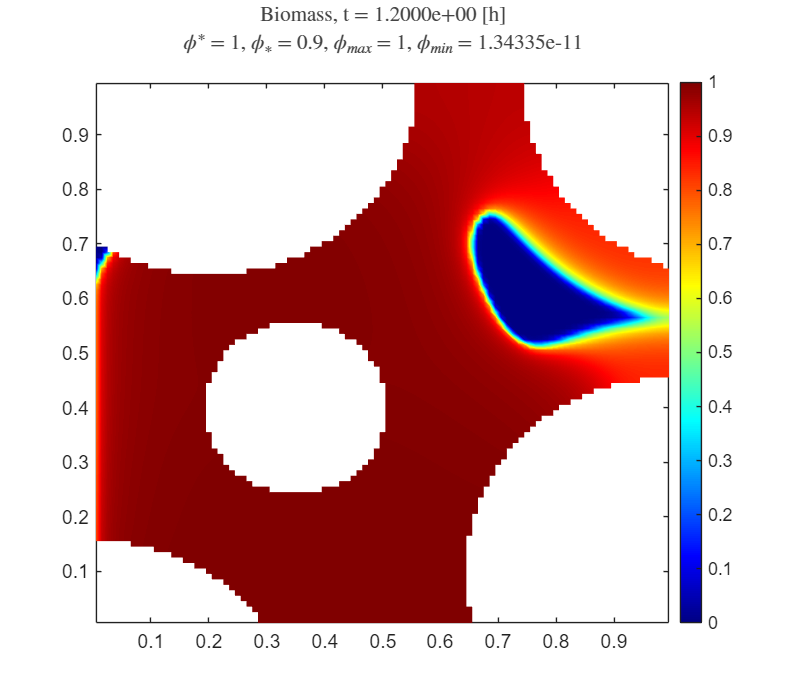}
    \caption{Comparison of stationary conditions (left) and active flow conditions (right)  }
    \label{fig:comparison}
\end{figure}
\begin{figure}[H]
    \centering
    \includegraphics[width=0.34\linewidth]{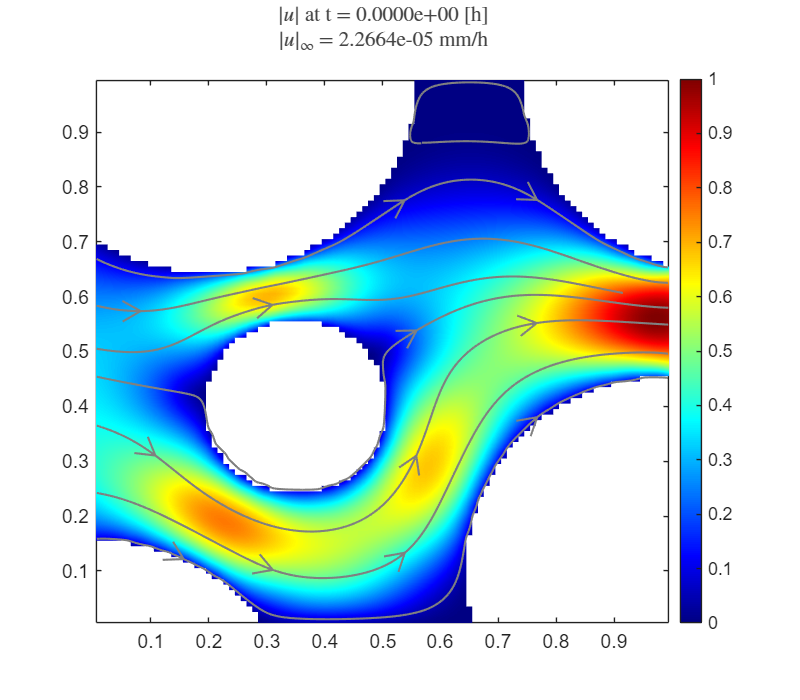}
     \includegraphics[width=0.34\linewidth]{Pic/u_tn0_dir3.png}
     \includegraphics[width=0.34\linewidth]{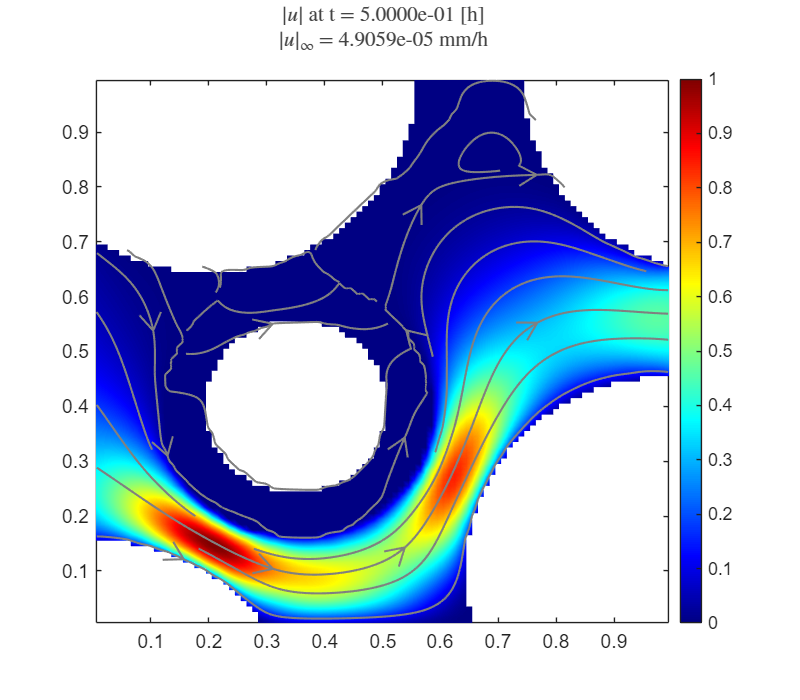}
     \includegraphics[width=0.34\linewidth]{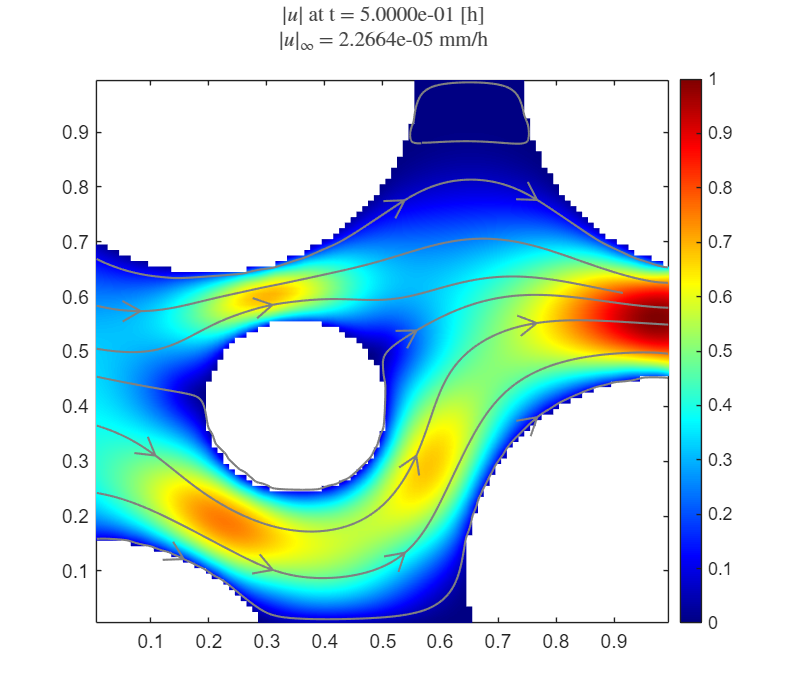}
     \includegraphics[width=0.34\linewidth]{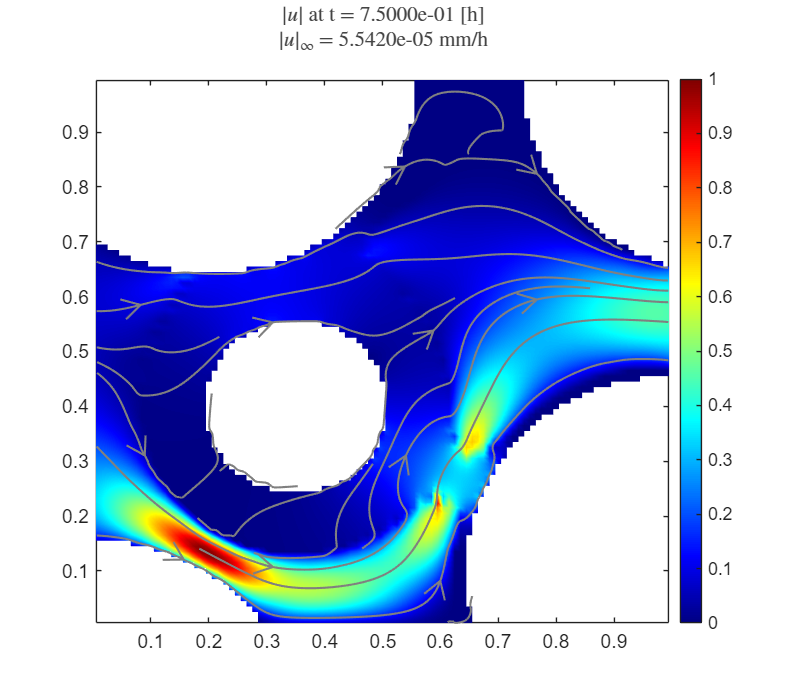}
      \includegraphics[width=0.34\linewidth]{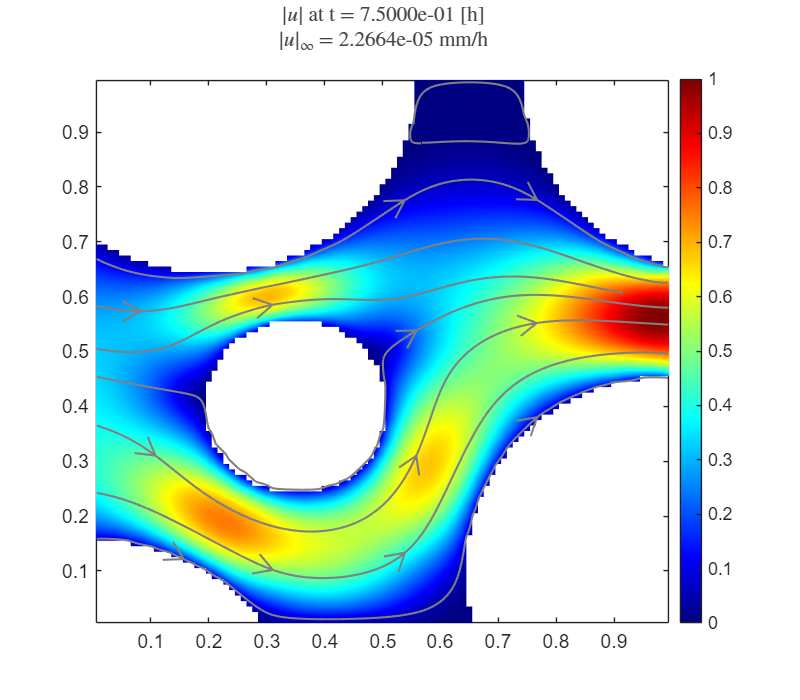}
    \includegraphics[width=0.34\linewidth]{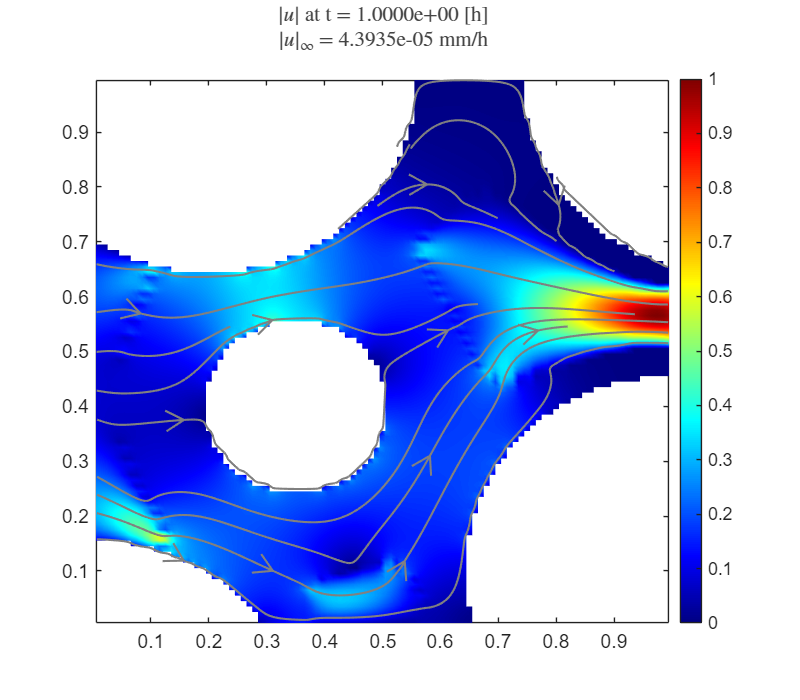}
    \includegraphics[width=0.34\linewidth]{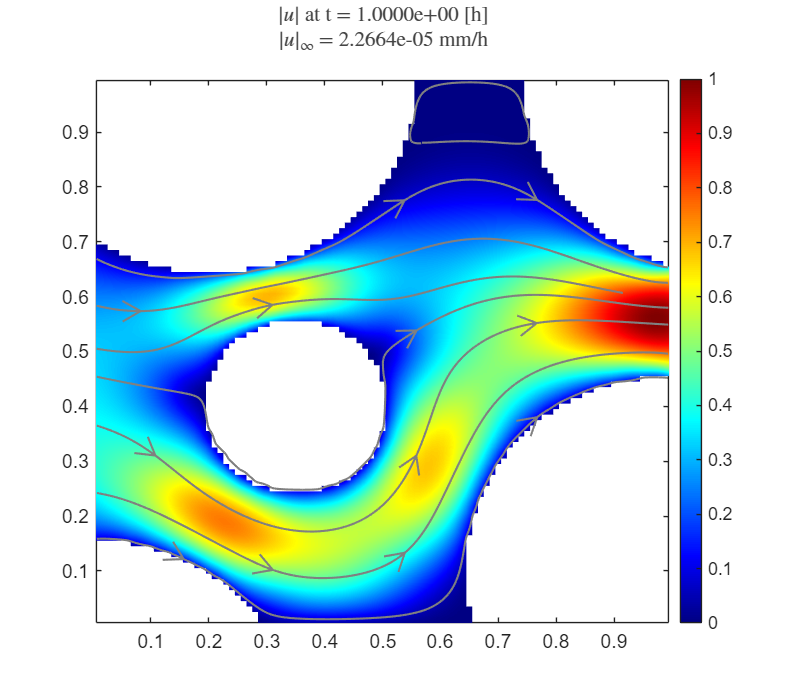}
    \caption{The effect of biofilm permeability on the flow; Left: $k_b=10^{-15}$. Right: $k_b=10^{-5}.$}
    \label{fig:placeholder}
\end{figure}
\myskip{
\begin{figure}[H]
    \centering
    \includegraphics[width=0.35\linewidth]{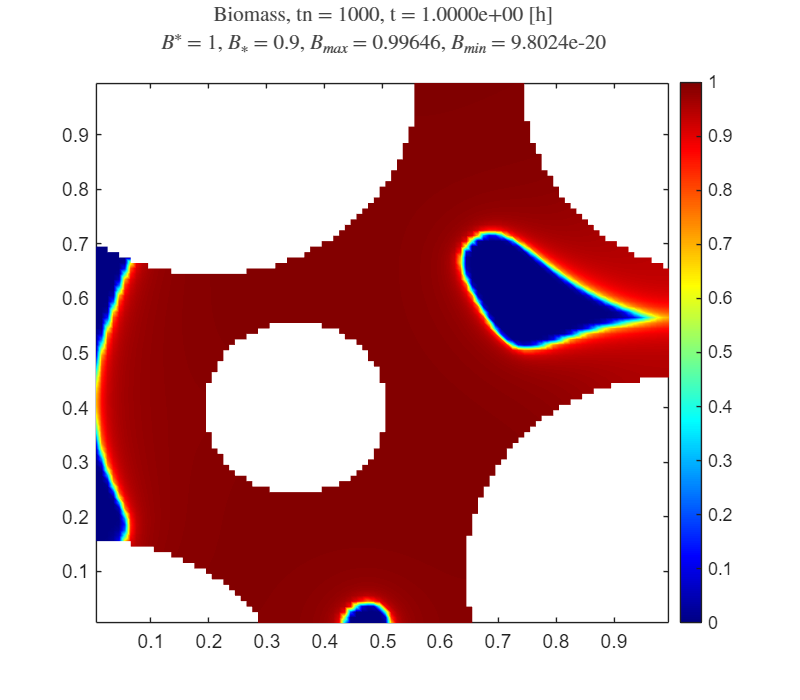}
    \includegraphics[width=0.35\linewidth]{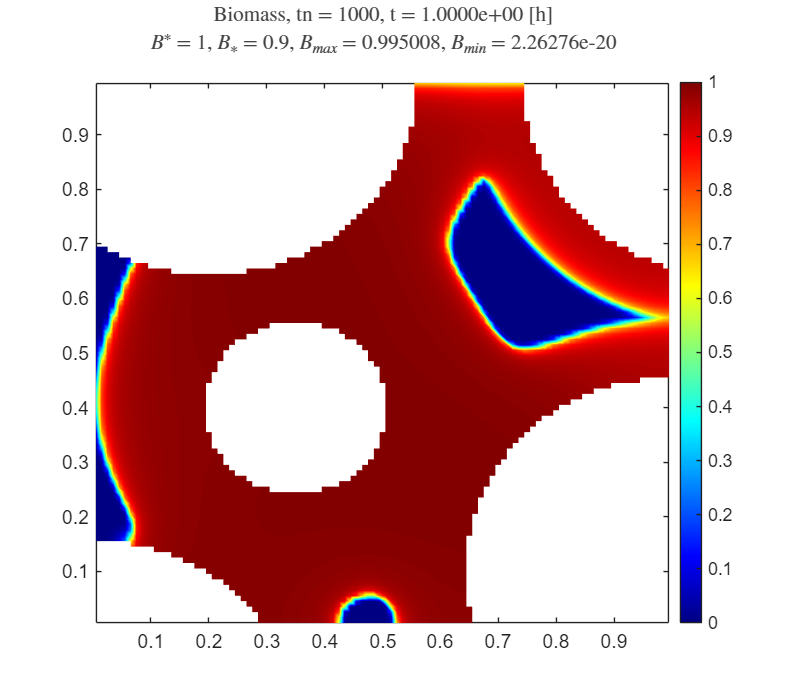}
    \caption{Biofilm growth after 1 hour with $B_{\mathrm{init}}=0.8B^\ast$.
    Left: static domain; right: dynamic domain.}
    \label{fig:case 1}
\end{figure}
}
\myskip{
\begin{figure}[H]
    \centering
    \includegraphics[width=0.3\linewidth]{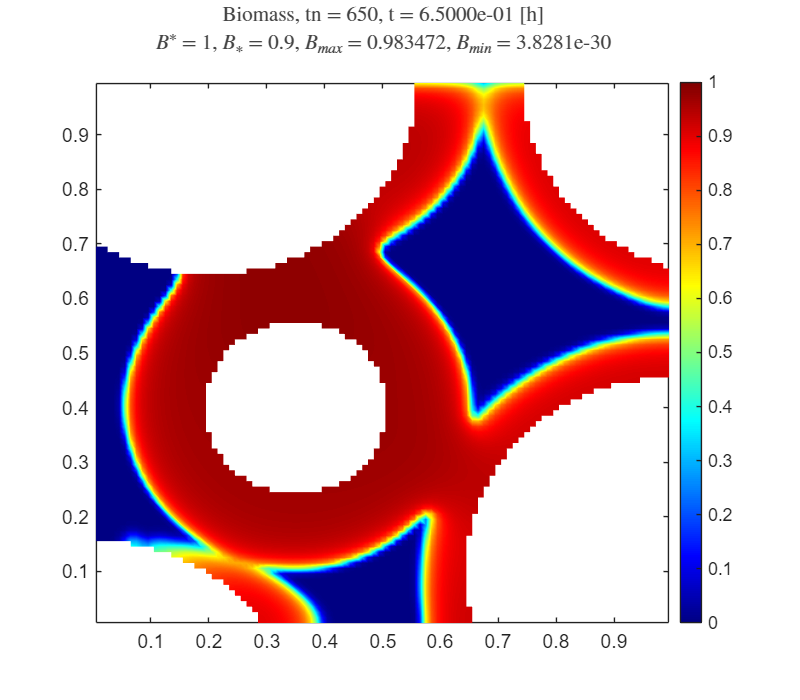}
    \includegraphics[width=0.3\linewidth]{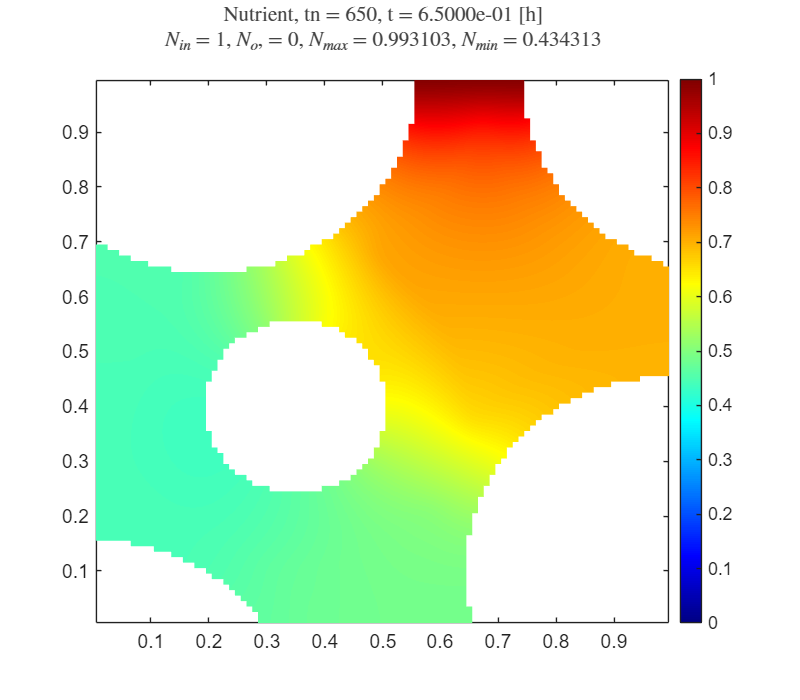}
    \includegraphics[width=0.3\linewidth]{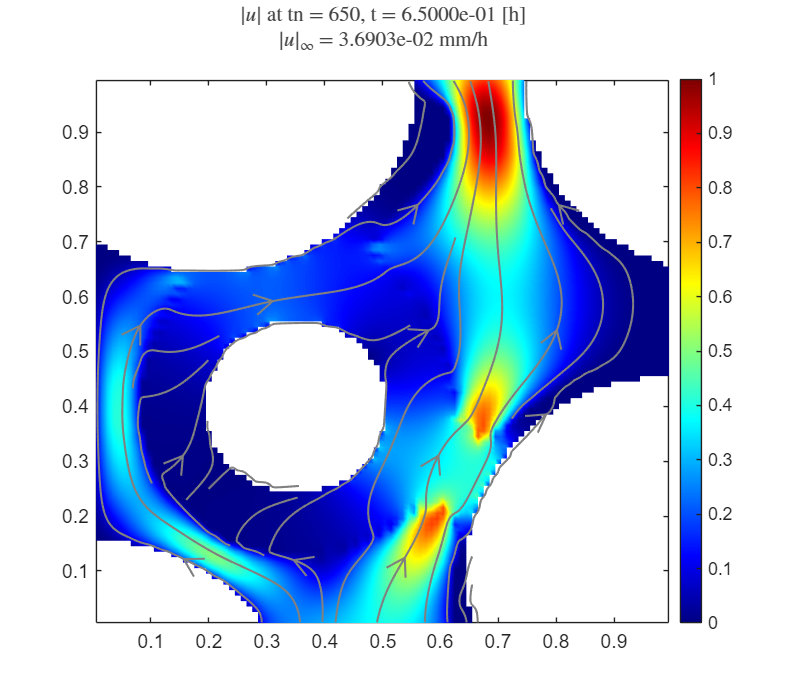}
    \includegraphics[width=0.3\linewidth]{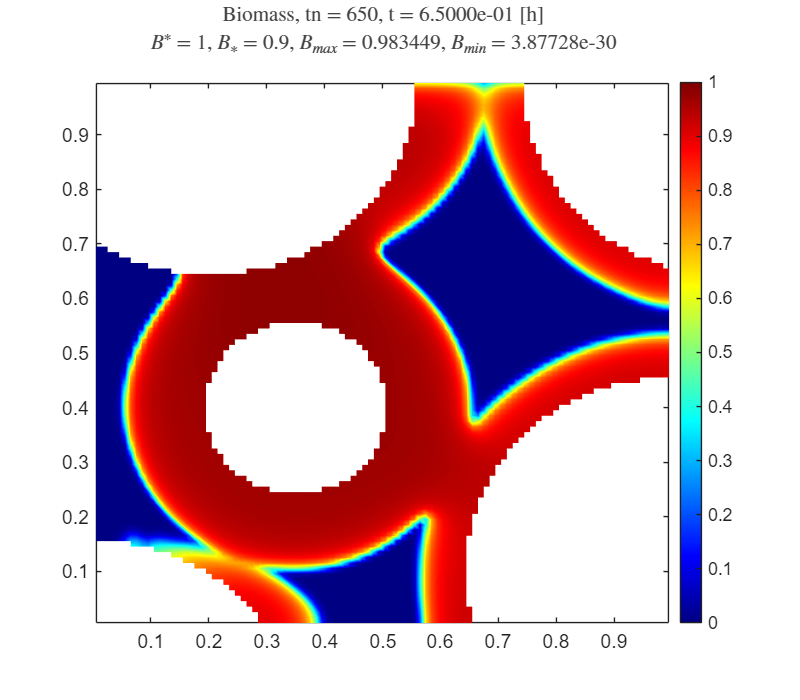}
    \includegraphics[width=0.3\linewidth]{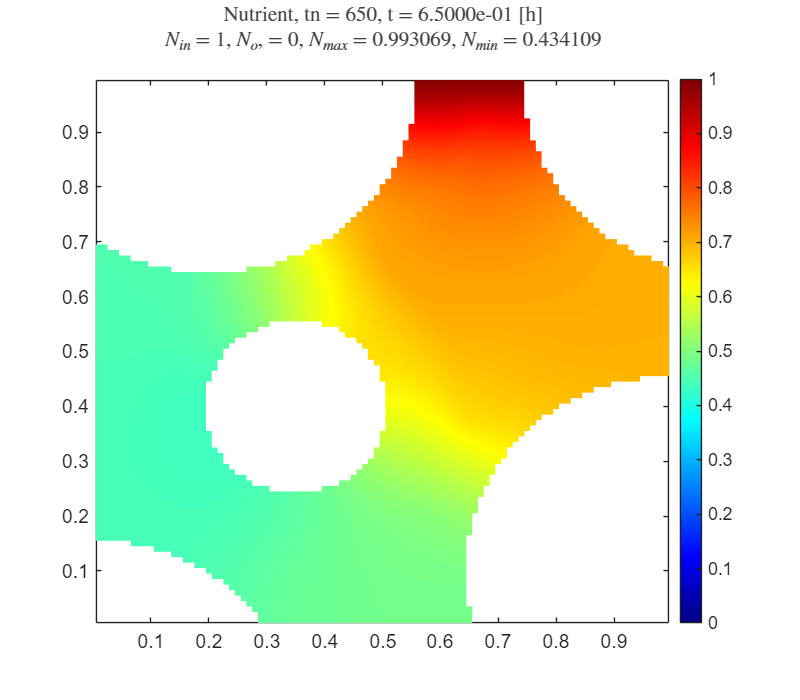}
    \includegraphics[width=0.3\linewidth]{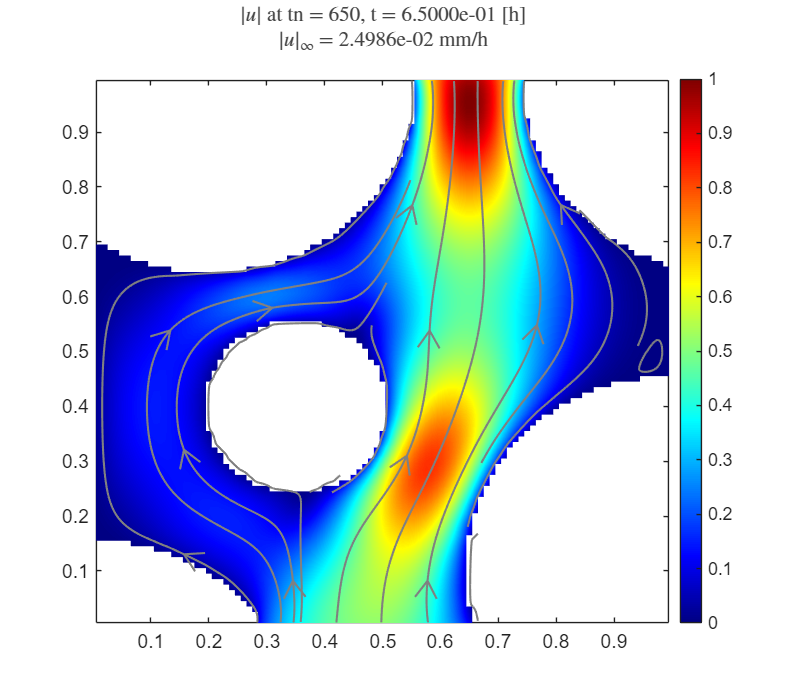}
    \includegraphics[width=0.3\linewidth]{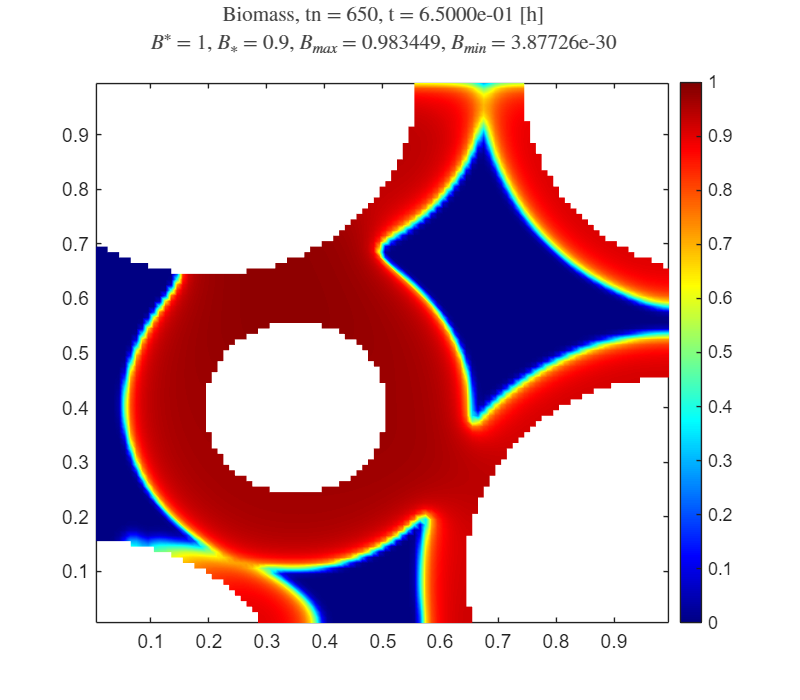}
    \includegraphics[width=0.3\linewidth]{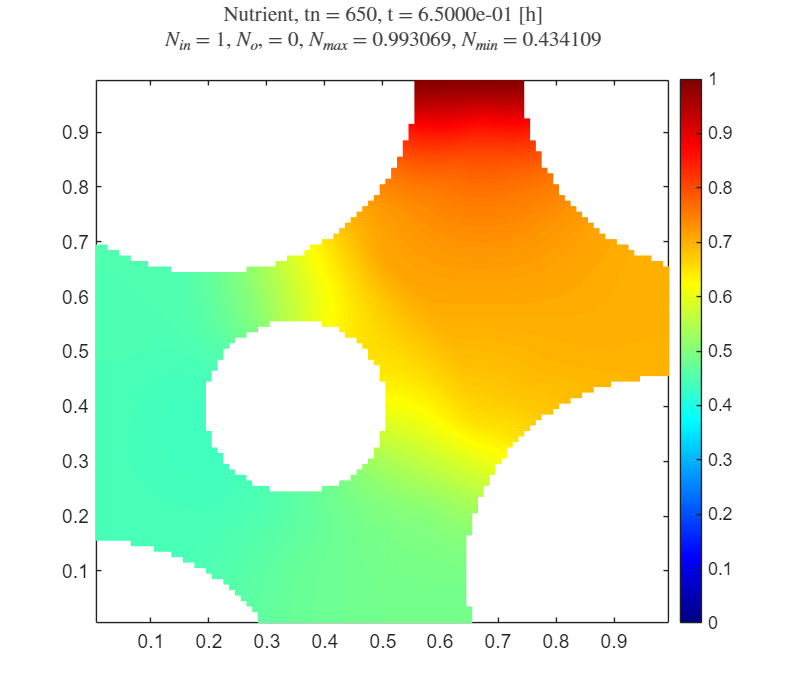}
    \includegraphics[width=0.3\linewidth]{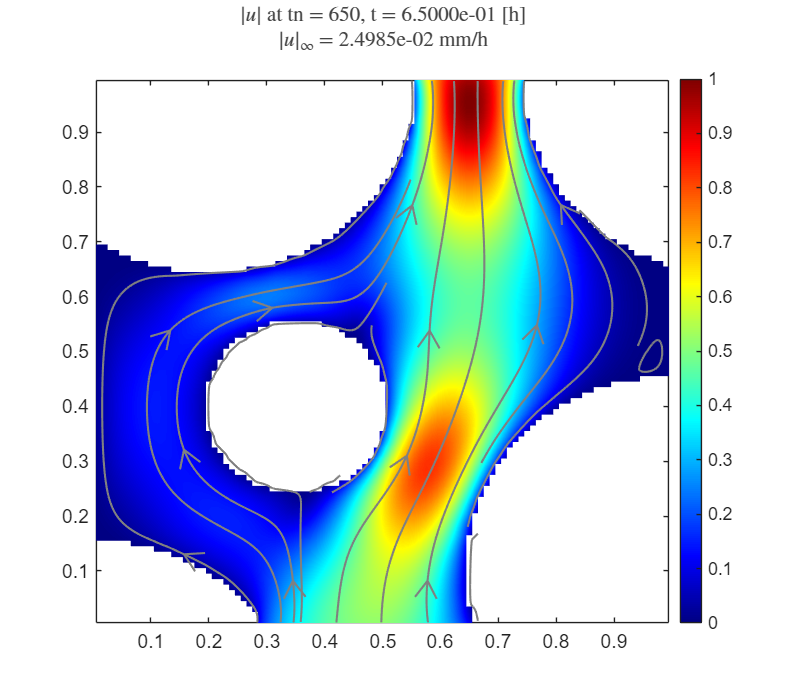}
    \caption{The effect of biofilm permeability on biofilm and flow with
    $B_{\mathrm{init}}=0.8B^\ast$. From left to right: biofilm, nutrient, flow.
    From top to bottom: $k_b=10^{-12}$, $k_b=10^{-5}$, $k_b=10^{-1}$.}
    \label{fig:case 2}
\end{figure}
}
\begin{figure}[H]
    \centering
    \includegraphics[width=0.34\linewidth]{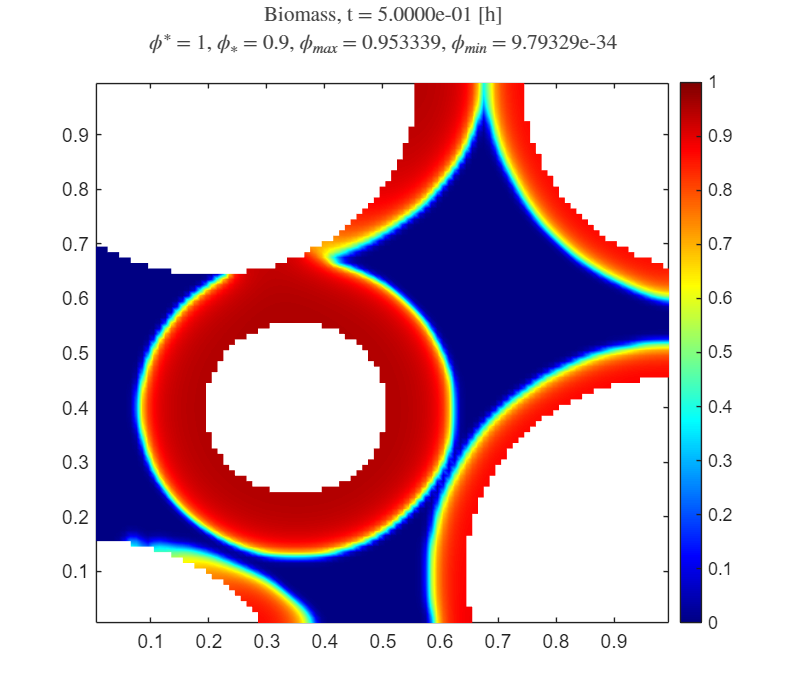}
    \includegraphics[width=0.34\linewidth]{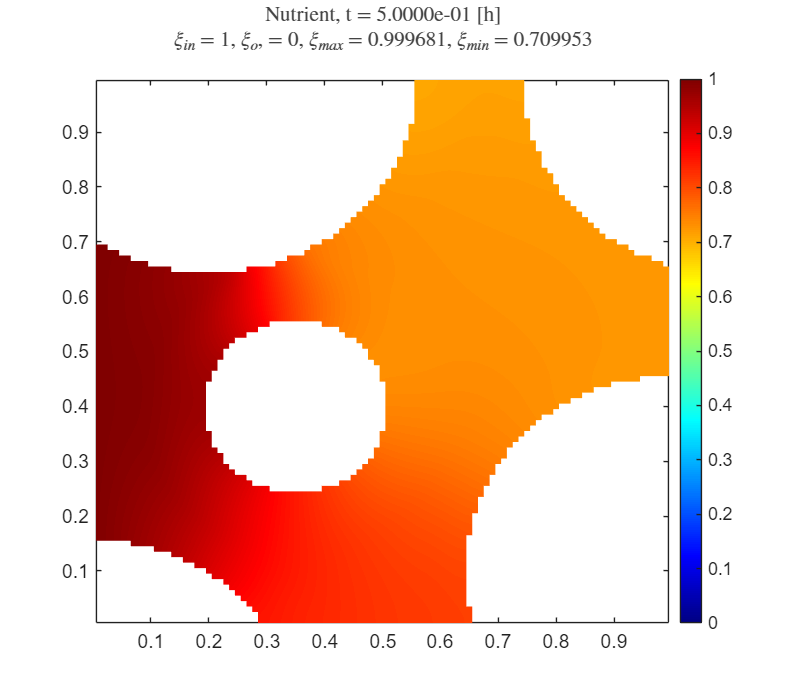}
    \includegraphics[width=0.34\linewidth]{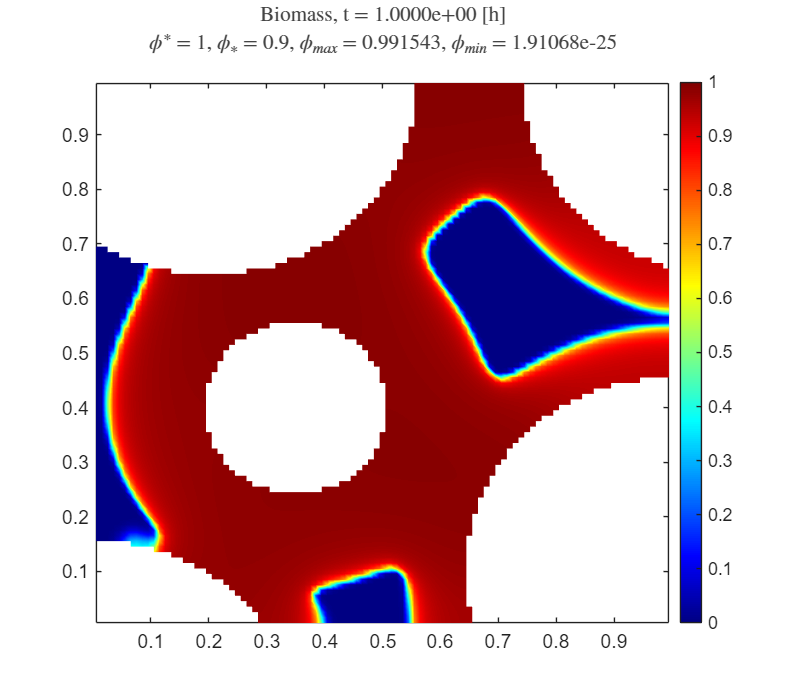}
    \includegraphics[width=0.34\linewidth]{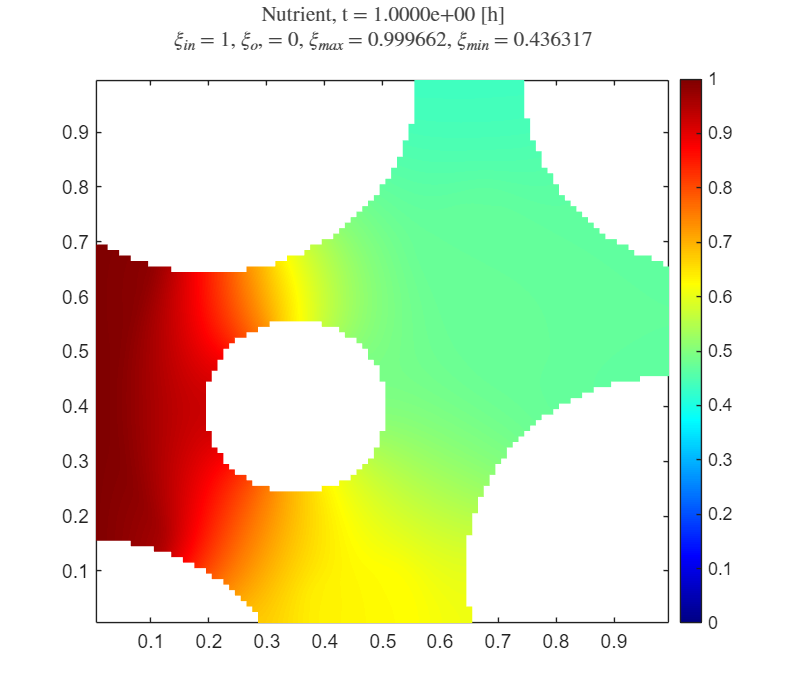}
     \includegraphics[width=0.34\linewidth]{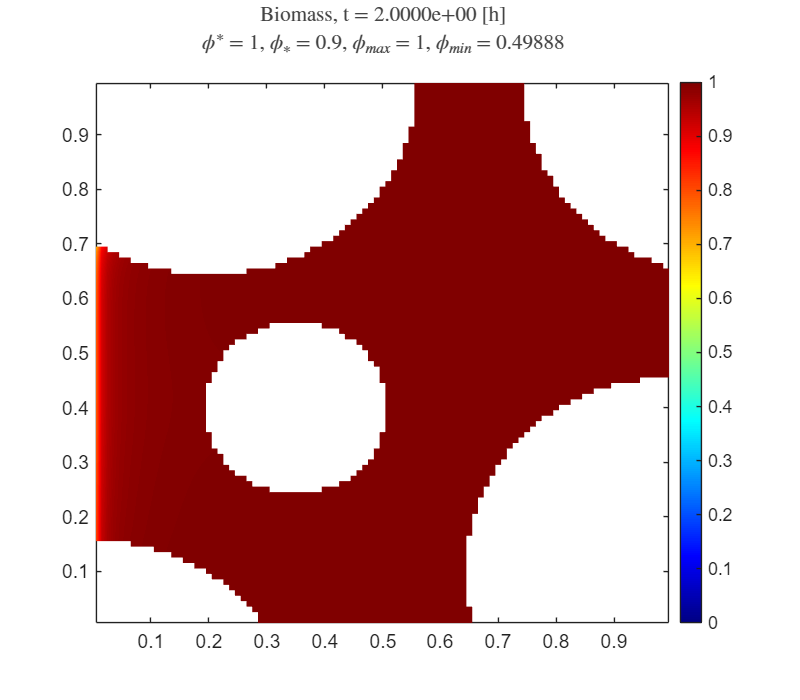}
    \includegraphics[width=0.34\linewidth]{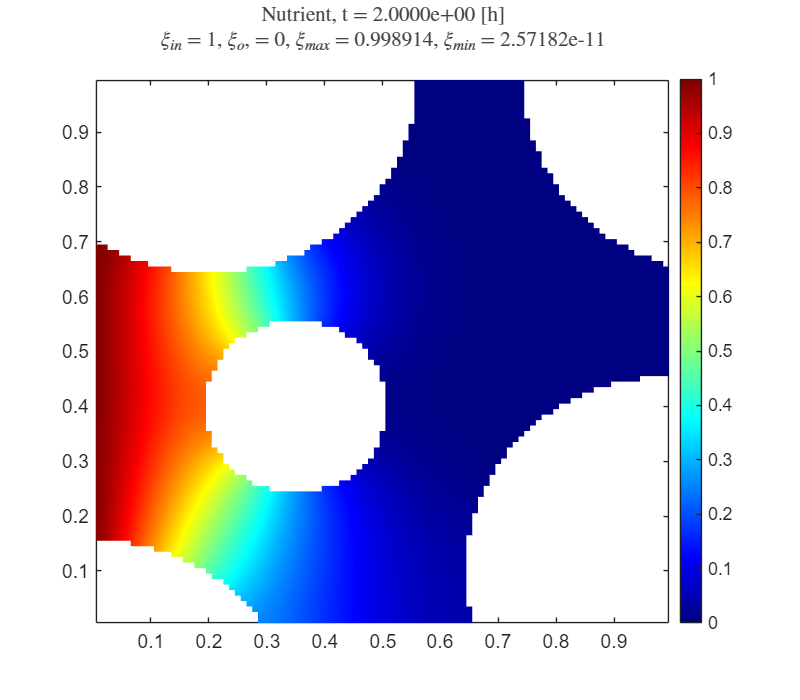}
     \includegraphics[width=0.34\linewidth]{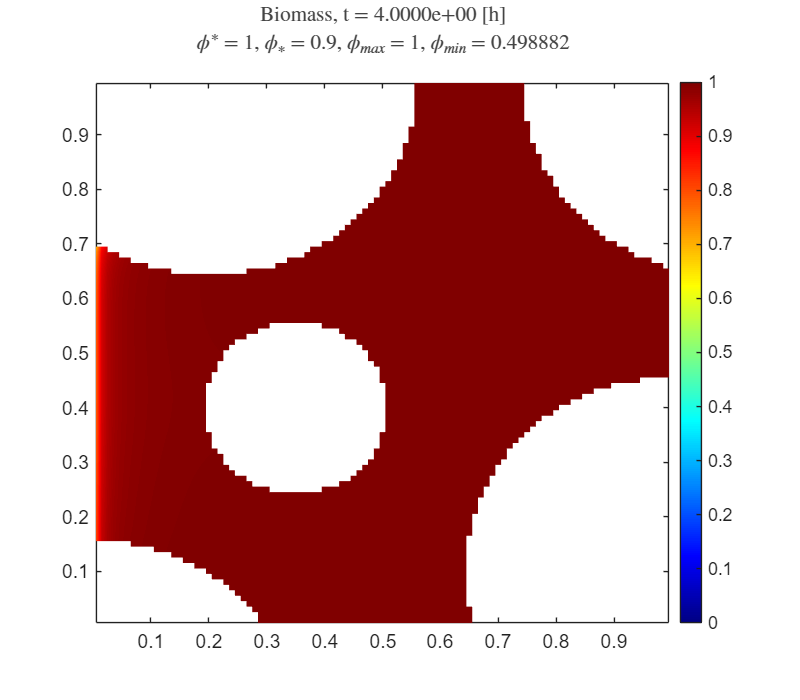}
    \includegraphics[width=0.34\linewidth]{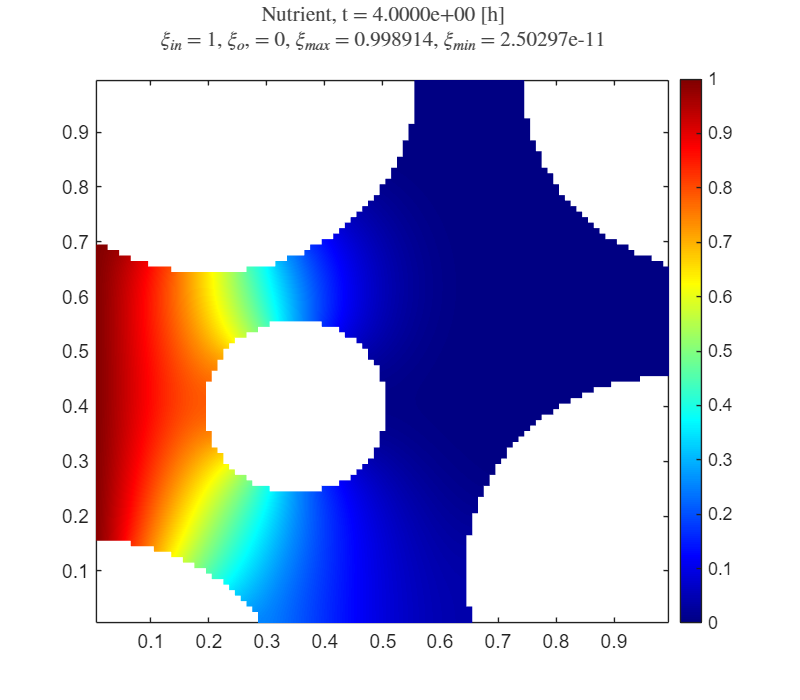}
    \caption{The evolution of biofilm and nutrient over time with $\phi_0=0.7$, $k_b=10^{-5}.$
   }
    \label{fig:case 3}
\end{figure}
\myskip{
\begin{figure}[H]
    \centering
    \includegraphics[width=0.3\linewidth]{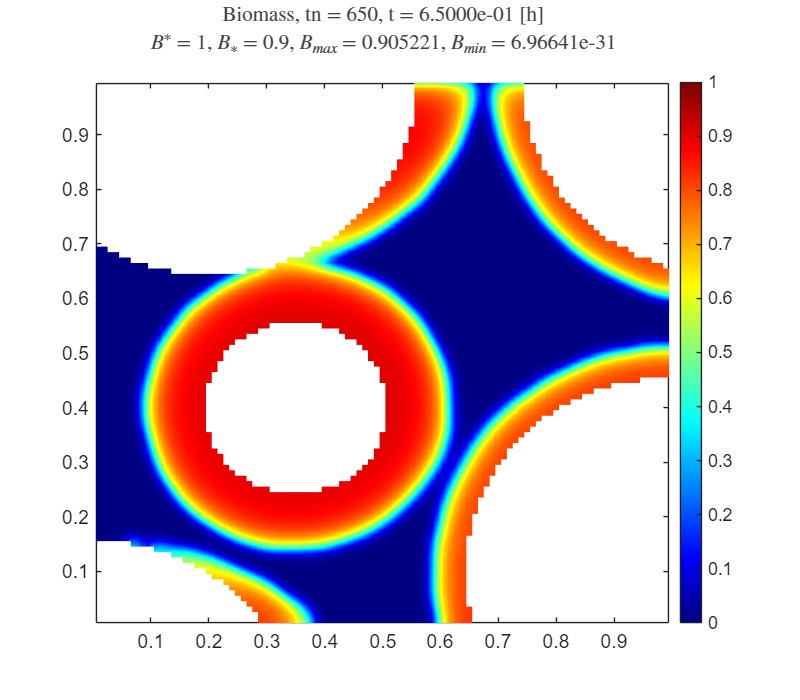}
    \includegraphics[width=0.3\linewidth]{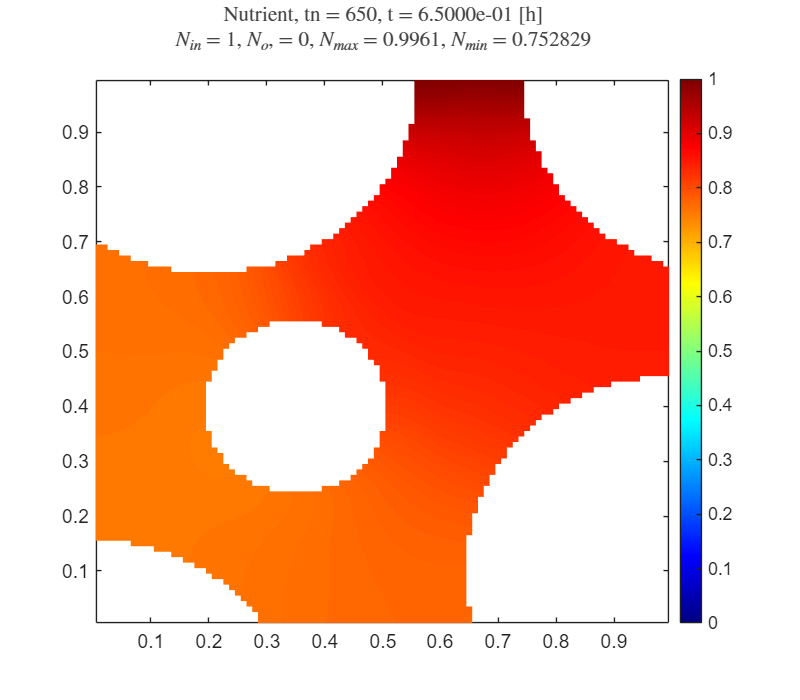}
    \includegraphics[width=0.3\linewidth]{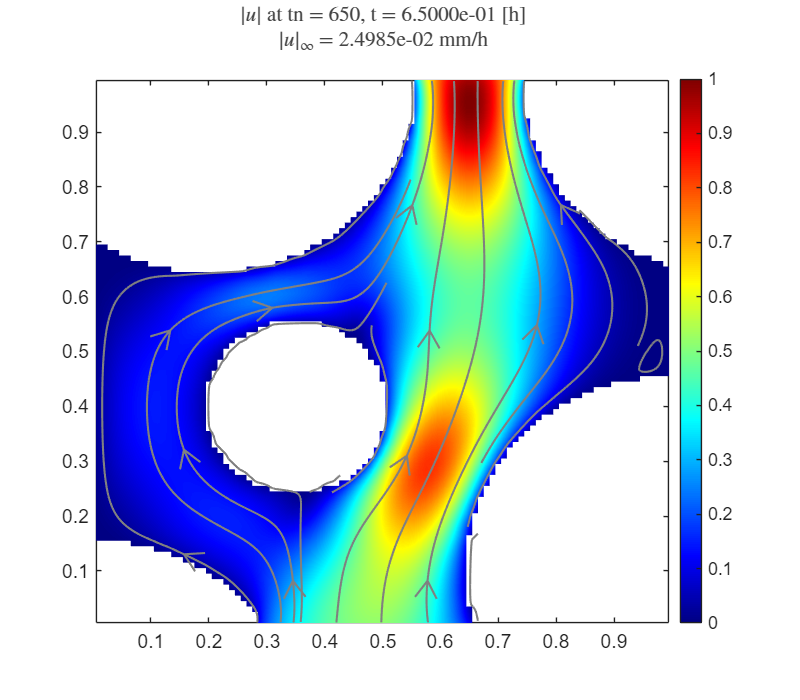}
    \caption{Biofilm, nutrient, and flow with $B_{\mathrm{init}}=0.5B^\ast$ and
    $k_b=10^{-5}$. From left to right: biofilm, nutrient, flow.}
    \label{fig:case 3 alt}
\end{figure}
}

\begin{figure}
    \centering
    \includegraphics[width=0.7\linewidth]{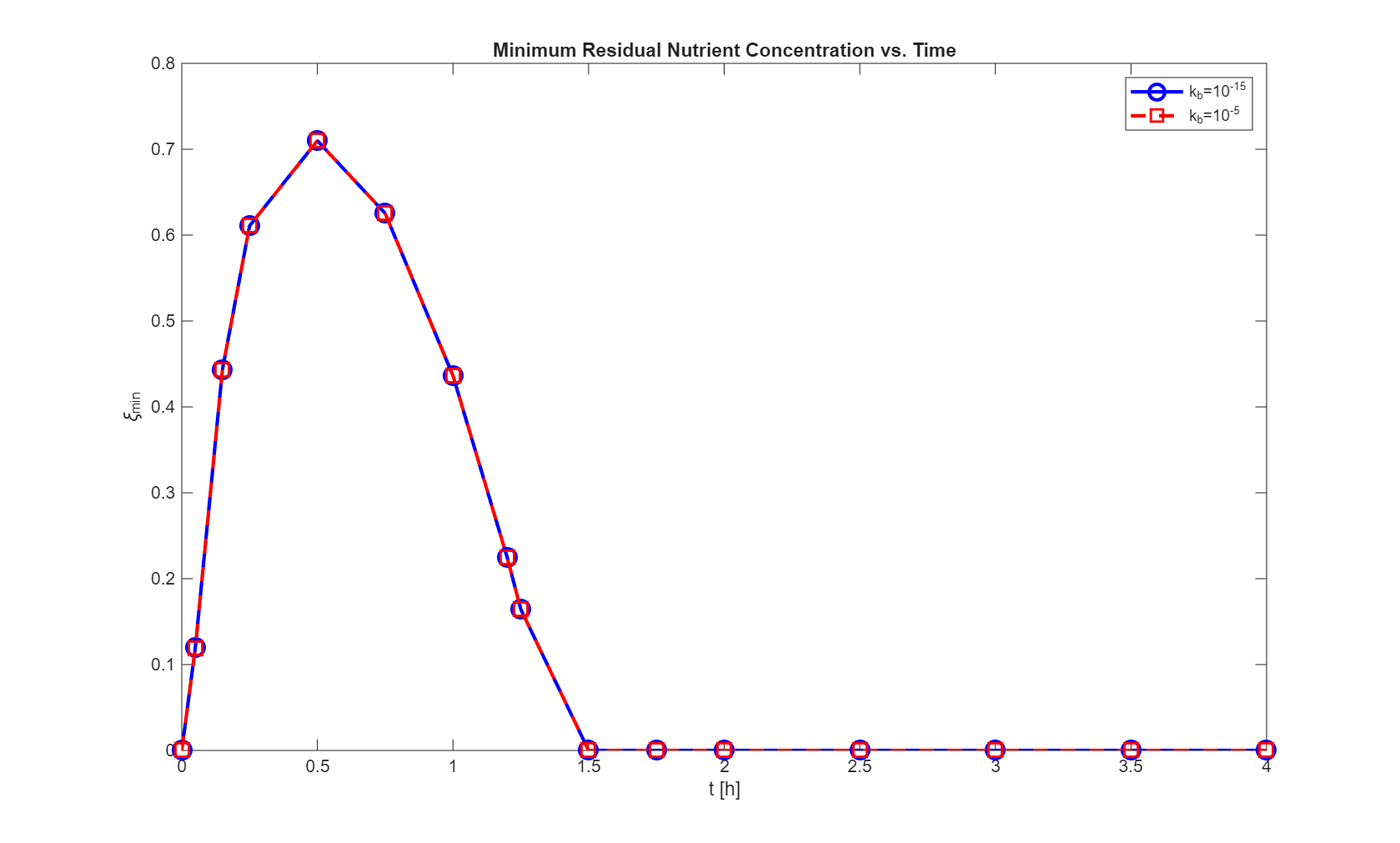}
    \caption{Minimum nutrient concentration over time: the nutrient remains nonnegative throughout the simulation, in agreement with Lemma~\ref{lem:xi-nonneg}.}
    \label{fig:permeability}
\end{figure}


\end{document}